\documentclass[twoside,11pt]{article}
\usepackage{subfigure,epsfig,amsfonts}
\usepackage[square]{natbib}
\usepackage[margin=1in]{geometry}
\usepackage{amsmath}
\usepackage{appendix}

\newtheorem{thm}{Theorem}[section]
\newtheorem{prop}[thm]{Proposition}
\newtheorem{lem}[thm]{Lemma}
\let\oldlem\lem
\renewcommand{\lem}{\oldlem\normalfont}

\newtheorem{exmp}{Example}[section]

\newtheorem{assump}[thm]{Assumption}
\newtheorem{proof}[thm]{Proof}

\newtheorem{remark}[thm]{Remark}

\begin{document}

%%
%% The title of the paper goes here.  Edit to your title.
%%

\title{Learning high-dimensional graphical models with regularized quadratic scoring}

\author{Eric Janofsky \\
	   ejanofsky@uber.com \\
       Uber Technologies, Inc.\\
       1400 Broadway\\
       New York, New York \\
       10018}

\maketitle
\abstract{
Pairwise Markov Random Fields (MRFs) or undirected graphical models are parsimonious representations of joint probability distributions. Variables correspond to nodes of a graph, with edges between nodes corresponding to conditional dependencies. Unfortunately, likelihood-based learning and inference is hampered by the intractability of computing the normalizing constant.
This paper considers an alternative scoring rule to the log-likelihood, which obviates the need to compute the normalizing constant of the distribution. We show that the rule is a positive-definite quadratic function of the natural parameters. We  optimize the sum of this scoring rule and a sparsity-inducing regularizer. For general continuous-valued exponential families, we provide theoretical results on parameter and edge consistency. As a special case we detail a new approach to sparse precision matrix estimation whose theoretical guarantees match that of the graphical lasso of \cite{yuan2007model}, with faster computational performance than the glasso algorithm of \cite{yuan2010high}. We then describe results for model selection in the nonparametric pairwise graphical model using exponential series. The regularized score matching problem is shown to be a convex program; we provide scalable algorithms based on consensus alternating direction method of multipliers (\cite{boyd2011distributed}) and coordinate-wise descent.}
\section{Introduction}

Undirected graphical models are an invaluable class of statistical models. They have been used successfully in fields as diverse as biology, natural language processing, statistical physics and spatial statistics. The key advantage of undirected graphical models is that its joint density may be factored according to the cliques of a graph corresponding to the conditional dependencies of the underlying variables. The go-to approach for statistical estimation is the method of maximum likelihood (MLE). Unfortunately, with few exceptions, MLE is intractable for high-dimensional graphical models, as it requires computation of the normalizing constant of the joint density, which is a $d$-fold convolution. Even exponential family graphical models \citep{wainwright2008graphical}, which are the most popular class of parametric models, are generally non-normalizable, with a notable exception being the Gaussian graphical model. Thus the MLE must be approximated.  State of-the-art methods for graphical structure learning avoid this problem by performing neighborhood selection \citep{yang2012graphical,meinshausen2006high,ravikumar2010high}. However, this approach only works for special types of pairwise graphical models whose conditional distributions form GLMs. Furthermore, these procedures do not by themselves produce parameter estimates.

In this work we demonstrate a new method for graph structure learning and parameter estimation based on minimizing the regularized Hyv\"arinen score of the data. It works for any continuous pairwise exponential family, as long as it follows some weak smoothness and tail conditions. Our method allows for multiple parameters per vertex/edge. We prove high-dimensional model selection and parameter consistency results, which adapt to the underlying sparsity of the natural parameters. As a special case, we derive a new method for estimating sparse precision matrices with very competitive estimation and graph learning performance. We also consider how our method can be used to do model selection for the general nonparametric pairwise model by choosing the sufficient statistics to be basis elements with degree growing with the sample size. We show our method can be expressed as a second-order cone program, and which we provide scalable algorithms based on ADMM and coordinate-wise descent.
\section{Background}

\subsection{Graphical Models}
Suppose $X=(X_1,\ldots,X_d)$ is a random vector with each entry having support $\mathcal{X}_i$, $i=1,\ldots,d$. Let $G=(V,E)$ be an undirected graph on $d$ vertices corresponding to the elements of $X$. An undirected graphical model or Markov random field is the set of distributions which satisfy the Markov property or condition independence with respect to $G$. From the Hammersley-Clifford theorem, if $X$ is Markov with respect to $G$, the density of $X$, $p$ can be decomposed as
\begin{align}
	p(x) \propto \exp\left\{\sum_{c\in \text{cl}(G)} \psi_c(x_c)\right\},
\end{align}
where $\text{cl}(G)$ is the collection of cliques of $G$. The pairwise graphical model supposes the density can be further factored according to the edges of $G$,
\begin{align}
	p(x) \propto \exp\left\{\sum_{i,j\in V,i\leq j}\psi_{ij}(x_i,x_j)\right\}.
\end{align}

%Here $\theta_i=\left(\theta_i^1,\ldots\theta_i^{m}\right)^\top$ and $\theta_{ij}=\left(\theta_{ij}^1,\ldots\theta_i^{m}\right)^\top$ are vectors of natural parameters, and $\phi=\left(\phi_i^1,\ldots\phi_i^{m}\right)^\top$ and $\phi_{ij}=\left(\phi_{ij}^1,\ldots\phi_{ij}^{m}\right)^\top$ are vectors of sufficient statistics. 

For a pairwise exponential family, we parametrize $\psi_i,\psi_{ij}$ by
\begin{align}
	\psi_{ii}(x_i) &:= \sum_{u\leq m} \theta_i^u \phi_i^u(x_i), & i\in V, \\
	\psi_{ij}(x_i,x_j) &:= \sum_{u\leq m} \theta_{ij}^u \phi_{ij}^u(x_i,x_j), & (i,j)\in E.
\end{align}
Here $m$ denotes the maximum number of statistics per edge or vertex. We denote $\theta$ to be the vectorization of the parameters, $\theta:=(\theta_{11}^\top,\ldots,\theta_{d1}^\top,\theta_{22}^\top,\ldots,\theta_{dd}^\top)^\top$.
%Here the $\theta$ is a vector of \emph{natural parameters} and $\phi$ a vector of \emph{sufficient statistics}. To ensure identifiability we require symmetry, $\theta_{ij}^u=\theta_{ji}^u$, and $\phi_{ij}^u=\phi_{ji}^u$. We allow the possibility of multiple statistics for a given vertex or edge. We assume each vertex and edge has no more than $m$ sufficient statistics, and for notational convenience if a group has fewer than $m$ statistics, we pad these vectors with zeros. Denoting the $md\times d$ matrices 
%\begin{align}
%\bar{\theta}=\left[\begin{array}{cc}
%	\theta^1 \\
%	\vdots \\
%	\theta^m
%	\end{array}\right], &&
%	\bar{\phi} = \left[\begin{array}{cc}
%	\phi^1 \\
%	\vdots \\
%	\phi^m
%	\end{array}\right], \label{matform}
%\end{align}
%
%where we use the notation $\theta^1 := [\theta^1_{ij}]$ to denote the symmetric $d\times d$ matrix of parameters. We relate the two forms of notation by \eqref{matform}, by denoting $\theta=\text{vec}(\bar{\theta})$.

%We may alternatively express thedensity as
%\begin{align}
%	p(x) \propto \exp\left\{ \text{trace}\left\{\frac{1}{2} \bar{\theta}^\top\bar{\phi}(x)\right\}\right\}.\label{matrixform}
%\end{align}
\subsection{Scoring Rules}
A \emph{scoring rule} \citep{dawid2005geometry} $S(x,Q)$ is a function which measures the predictive accuracy of a distribution $Q$ on an observation $x$. A scoring rule is \emph{proper} if $\mathbb{E}_p[S(X,Q)]$ is uniquely minimized at $Q=P$. When $Q$ has a density $q$, we equivalently denote the scoring rule $S(X,q)$. A \emph{local} scoring rule only depends on $q$ through its evaluation at the observation $x$. A proper scoring rule induces an entropy
\begin{align}
	H(p) = \mathbb{E}_{p}\left[S(X,p)\right],
\end{align}
as well as a divergence
\begin{align}
	D(p,q) = \mathbb{E}_p \left[ S(X,q) - S(X,p)\right].
\end{align}
An optimal score estimator is an estimator which minimizes the empirical score 
\begin{align}
\frac{1}{n}\sum_{r=1}^n S(X^r,q),
\end{align}
 over some class of densities.
%
%The \emph{optimal (regularized) score estimator} finds the density $q\in\mathcal{P}$ which minimizes the regularized expected score:
%
%\begin{align}
%	\widehat{q} &:= \text{argmin}\left\{\mathbb{E}_p[S(X,q)] +\lambda\mathcal{R}(q) \right\}=\text{argmin}_{q\in\mathcal{P}} \left\{D(p,q) + \lambda\mathcal{R}(q)\right\}
%\end{align}
%
%where $\mathcal{R}$ is a regularizer. Given i.i.d. samples $X^1,\ldots,X^n \sim Q$, we replace the population expectation with the expectation with respect to the empirical distribution:
%
%\begin{align}
%	\text{argmin}\left\{ \frac{1}{n}\sum_{j=1}^n S(X^j,q) + \lambda\mathcal{R}(q)\right\}.
%\end{align}

\begin{exmp}
The \emph{log score} takes the form $l(x,q):=-\log q(x)$. The corresponding entropy is the Shannon entropy $H(p):=-\mathbb{E}_p\left[\log p\right]$, its corresponding divergence is the \emph{Kullback-Leibler Divergence} $\text{KL}(p \mid q) =\mathbb{E}_{X\sim p}\left[\log\frac{p(X)}{q(X)}\right]$ and the optimal score estimator is the maximum likelihood estimator. It is a proper and local scoring rule.
\end{exmp}

\begin{exmp}
Consider the \emph{Bregman score},
\begin{align}
b(x,q):=-g'(q(x))-\intop \rho(dy)\left(g(q(y))-q(y)g'(q(y))\right), \label{bregman}
\end{align}
where $g:\mathbb{R}^+\rightarrow\mathbb{R}$ is a convex, differentiable function and $\rho$ is some baseline measure. The corresponding entropy is $H(p)=-\intop \rho(dy)g(p)$ and divergence 
\begin{align}
D(p,q)=\intop \rho(dy)\left(g(p)-\left(g(q)+g'(q)(p-q)\right)\right).
\end{align}
%When $g(x)=\log(x)$, after removing the constant term, $b$ has the form
%\begin{align}
%	b(x,q)& = -\frac{1}{q(x)} -\intop \rho(dy) \log(q(y)) \\
%	&= -e^{-f(x)} - \intop \rho(dy) f(y),
%\end{align}
%where $f=\log q$. This is a proper scoring rule \citep{dawid2014theory}, but it is not local because it depends on values of $q$ besides the observation $x$. An estimation procedure for nonparametric graphical models using smoothing splines was based on this scoring rule in \citep{jeon2006effective}.
\end{exmp}
\subsection{Hyv{\"a}rinen Score}
Consider densities $q$ which are twice continuously differentiable over $\mathcal{X}=\mathbb{R}^d$ and satisfy
\begin{align}
	%&\mathbb{E}_q\left[\Vert \nabla\log q(X)\Vert^2\right]<\infty, \\
	%&\mathbb{E}_q\left[\Delta\log q(X)\right] <\infty,\\
	& \Vert p(x) \nabla \log q(x)\Vert\rightarrow 0, \text{ for all } \Vert x \Vert \rightarrow \infty. \label{boundary}
\end{align}	
where $X\sim p$. Consider the scoring rule
\begin{align}
	h(x,q) &=\frac{1}{2}\Vert\nabla \log q(x)\Vert^2_2 + \Delta \log q(x),
\end{align}
where $\nabla$ denotes the gradient operator and  $\Delta$ is the operator 
\begin{align}
	\Delta \phi(x) = \sum_{i\in V} \frac{\partial^2\phi(x)}{\partial x_i^2}.
\end{align}
This is a proper and local scoring rule \citep{parry2012proper}. Using integration by parts, it can be shown it induces the \emph{Fisher divergence}:
\begin{align}
	\text{F}(p \mid q) = \mathbb{E}_{X\sim p}\left[ \bigg\Vert \nabla \log \frac{p(X)}{q(X)} \bigg\Vert_2^2 \right].
\end{align}

The optimal score estimator is called the \emph{score matching} estimator \citep{hyvarinen2005estimation,hyvarinen2007some}. The Hyv{\"a}rinen score is homogeneous in $q$ \citep{parry2012proper}, so that it does not depend on the normalizing constant of $q$, which for multivariate exponential families is typically intractable. Second, for natural exponential families the objective of the optimal score estimator is quadratic, so the estimating equations corresponding to score matching are linear in the natural parameters \citep{forbes2014linear}. Maximum likelihood for exponential families generally involves a complex mapping from the sufficient statistics of the data to the natural parameters \citep{wainwright2008graphical,brown1986fundamentals}, necessitating specialized solvers.

\subsection{Score Matching for Exponential Families} \label{smef}
%Define the Jacobian of $\phi$ by
%
%\begin{align}
%	J(x):= \left[\frac{\partial\phi(x)}{\partial x_1},\ldots,\frac{\partial\phi(x)}{\partial x_d}\right],
%\end{align}
%
%
%For an exponential family with sufficient statistics $\phi$, the Hyv\"arinen score is given by
%\begin{align}
%	h(x,\theta) &= \frac{1}{2}\Vert J(x)^\top\theta\Vert^2 + \langle\theta,\Delta\phi(x)\rangle \\
%		&= \frac{1}{2} \theta^\top J(x) J(x)^\top \theta +\theta^\top \Delta\phi(x),
%\end{align}
%
%Since $h$ is proper, $\theta^*$ may be found to be the minimum of $\mathbb{E}_p[h(X,\theta)]$. Define
%\begin{align}
%	\Gamma^* &:= \mathbb{E}_p[J(X) J(X)^\top], \\
%	K^* &:= \mathbb{E}_p[\Delta\phi(X)],
%\end{align}
%After taking derivatives, we get the score equation
%
%\begin{align}
%	\Gamma^*\theta^* +K^* &= 0.
%\end{align}
%
%and assuming $\Gamma^*$ is strictly positive definite, $\theta^*=-(\Gamma^*)^{-1}K^*$. Thus, the natural parameters $\theta^*$ of the exponential family can be represented as the solution of  a linear equation of the statistics $\Gamma^*,K^*$; see also \citep{hyvarinen2007some,forbes2014linear}.

For a pairwise density, define $\phi_{\cdot,i}=\left(\phi_{1i}^\top,\ldots,\phi_{di}^\top\right)^\top$. For $i\in V$, denote $a_i(x) := \frac{\partial}{\partial x_i} \phi_{\cdot,i}$ and
\begin{align}
	(K(x))_{\cdot,i}:= \frac{\partial^2\phi_{\cdot,i}}{\partial x_i^2}.
\end{align}
%and denote $K(x)$ the vectorization: $K(x)=\left(K_{11}^\top,K_{12}^%\top,\ldots,K_{1d}^\top,\ldots,K_{dd}^\top\right)^\top$.
Taking derivatives,
\begin{align}
	\frac{\partial}{\partial x_i} \langle \phi(x), \theta\rangle & =  \left\langle \frac{\partial\phi_{\cdot,i}}{\partial x_i},\theta_{\cdot,i}\right\rangle ,
\end{align}
thus $h$ takes the form
\begin{align}
	h(x,\theta)
	&= \sum_{i\in V} \left( \frac{1}{2}\theta_{\cdot,i}^\top a_i(x)a_i(x)^\top\theta_{\cdot,i} +K_{\cdot,i}(x)^\top\theta_{\cdot,i}\right).
\end{align}
$h$ is a sum of $d$ positive-semidefinite quadratic forms, so it is also psd quadratic. Alternatively, we may write $h(x,\theta)=\theta^\top A(x)\theta + K(x)^\top \theta$, where $A(x)$ is a psd matrix with at most $2md$ non-zero entries per row, and $K(x)$ is a vector with $K_{ij}=\frac{\partial^2 \phi_{ij}}{\partial x_i ^2}+ 1\{i\not = j\}\frac{\partial^2 \phi_{ij}}{\partial x_j ^2}$. If we write $\tilde{\theta}=\left(\theta_{\cdot,1}^\top,\theta_{\cdot,2}^\top,\ldots,\theta_{\cdot,d}^\top\right)^\top$, where $\tilde{\theta}_{ij}=\tilde{\theta}_{ji}$, we may write the scoring rule as
\begin{align}
	h(x,\tilde{\theta}) = \frac{1}{2}\tilde{\theta}^\top \tilde{A}(x)\tilde{\theta} + \tilde{K}(x)^\top\tilde{\theta},\label{pairscore}
\end{align}
where $\tilde{A}(x)$ is a block-diagonal matrix,
\begin{align}
	\tilde{A}(x)= \left(\begin{array}{cccc}
		a_1(x)a_1(x)^\top &&&\\
		& a_2(x)a_2(x)^\top & & \\
		&& \ddots & \\
		&&& a_d(x)a_d(x)^\top
	\end{array}\right)
\end{align}
and $\tilde{K}=(K_{\cdot,1}^\top,\ldots,K_{\cdot,d}^\top)^\top$. We will alternate between these two equivalent representations of $h$ based on convenience.

%Note that even though $\Gamma$ is an $\frac{m^2d(d+1)}{2} \times \frac{m^2d(d+1)}{2}$ matrix, it has only $md$ nonzero entries per row. To clarify, consider the matrix form of the density \eqref{matrixform}. Denote $\bar{J}$ the $md\times d^2$ matrix
%
%\begin{align}
%	\bar{J}(x):= \left[\frac{\partial\bar{\phi}(x)}{\partial x_1},\cdots,\frac{\partial\bar{\phi}(x)}{\partial x_d}\right],
%\end{align}
%
%then for a pairwise density, we may alternatively write
%\begin{align}
%	h(x,\bar{\theta}) = \text{trace}\left\{\frac{1}{2}\bar{\theta}\bar{J}(x)\bar{J}(x)^\top \bar{\theta} + (\Delta\bar{\phi}(x))\bar{\theta}\right\}. \label{pairscore}
%\end{align}
%
%In this form, we denote the $md\times md$ matrices
%
%\begin{align}
%	\bar{\Gamma} := \mathbb{E}_p[\bar{J}(X)\bar{J}(X)^\top], \\
%	\bar{K} := \mathbb{E}_p[\Delta\bar{\phi}(X)].
%\end{align}
%
%For two matrices $A,B$, we have the relation $\text{trace}\left\{ ABA\right\}=\text{vec}(A)^\top(B\otimes I)\text{vec}(A)$. Applying this relation, we may observe that $\Gamma$ and $\bar{\Gamma}$ are related by
%
%\begin{align}
%	\Gamma=\bar{\Gamma}\otimes I_d.
%\end{align}

\begin{remark}[Bounded supports]

From the differentiability assumption we see that our derivations do not generally apply when $\mathcal{X}_i$ is a half-bounded or bounded support as the density may not be differentiable at the boundary. However, in \citep{hyvarinen2007some} a proper scoring rule was derived for half-bounded supports, which may be shown to have the same form as \eqref{pairscore}, after modifying slightly the formulas for $a_i(x),K(x)$. Here we derive a similar formula for densities on $[0,1]^d$.
\begin{prop}\label{boundedscorematch}
Consider random vectors taking values in $[0,1]^d$, with density $q$; suppose $X\sim p$. If $q$ is twice continuously differentiable and satisfies
\begin{align}
	&\Vert p(x) \nabla\log q(x) \otimes x(1-x) \Vert \rightarrow 0, \text{ for all } x\text{ approaching the boundary},
\end{align}
 where $\otimes$ denotes the tensor product $x\otimes y:= (x_1y_1,\ldots,x_dy_d)$, then 
 \begin{align}
 	h(x,q) &:= \frac{1}{2}\Vert \nabla \log q(x) \otimes x(1-x) \Vert_2^2 \notag\\
	&\qquad+ \sum_{i\in V}\left( -2(2x_i-1)x_i(1-x_i)\frac{\partial \log q(x)}{\partial x_i} + x_i(1-x_i)\frac{\partial^2 \log q(x)}{\partial x_i^2}\right),
 \end{align}
is a proper scoring rule. In  particular when $q$ is an exponential family with natural parameters $\theta$ and sufficient statistics $\phi$,
$h(x,\tilde{\theta})= \frac{1}{2} \tilde{\theta}^\top \tilde{A}(x)\tilde{\theta}+K(x)^\top\theta$ is a proper scoring rule, where
\begin{align}
	K(x)_{ij} &= -2(2x_i-1)x_i(1-x_i)\frac{\partial\phi_{ij}}{\partial x_i}+(x_i(1-x_i))^2\frac{\partial^2\phi_{ij}}{\partial x_i^2}, \\
	a_i(x) &= x_i(1-x_i)\frac{\partial \phi_{\cdot,i}}{\partial x_i}.
\end{align}
and $\tilde{A}(x)=\text{diag}(a_i(x)a_i(x)^\top)$.
\end{prop}

Thus, all of the results in this work may be effortlessly carried over to exponential families over bounded supports.

\end{remark}

\section{Previous Work}
\citep{sriperumbudur2013density} consider using the Hyv\"arinen score for density estimation in a reproducing kernel Hilbert space (RKHS). They consider the optimization for a density $q$,
\begin{align}
	\min_q \left\{\frac{1}{n}\sum_{k=1}^n h(X^k,q) + \frac{\lambda}{2}\Vert q\Vert_{\mathcal{H}}^2\right\},
\end{align}
where $\Vert\cdot\Vert_{\mathcal{H}}^2$ is the norm of the RKHS. After an application of the representer theorem \citep{kimeldorf1971some}, they show this may be expressed as a finite-dimensional quadratic program. They derive rates for convergence to the true density with respect to the Fisher divergence.

\citep{vincent2011connection} shows that the denoising autoencoder may be expressed as a type of score matching estimator, which they call \emph{denoising score matching}. Suppose that $\tilde{X}$ is a version of a sample $X$ which has been corrupted by Gaussian noise, so that its conditional distribution has the score $\partial \log q(\tilde{x}\mid x) = \frac{1}{\sigma^2} (x-\tilde{x})$. Suppose we seek to fit the corrupted data according to a density of the form
\begin{align}
	\log p(\tilde{x}\mid W,b,c) \propto -\frac{1}{\sigma^2}\left(\langle c,\tilde{x}\rangle-\frac{1}{2}\Vert \tilde{x}\Vert_2^2+\text{softplus}\left(\sum_j\langle W_j,\tilde{x}\rangle + b_j\right)\right),
\end{align}
where $\text{softplus}(x)=\max(0,x)$, then minimizing the Fisher divergence between the model density and $q(\tilde{x}\mid x)$ can be shown to be equivalent to minimizing
\begin{align}
	\mathbb{E}_{q(\tilde{x},x)}\left[\Vert W^\top\text{sigmoid}(W\tilde{X}+b)+c-X\Vert^2\right].
\end{align}
This is a simple denoising autoencoder with a single hidden layer, encoder $f(\tilde{x})=\text{sigmoid}(W\tilde{x}+b)$, and decoder $f'(y)=W^\top y+c$.

Score matching has also been used for learning natural image statistics \citep{kingma2010regularized,koster2009estimating}.

\section{Score Matching Estimator}
Define the statistics
\begin{align}
	\widehat{\Gamma} = \frac{1}{n}\sum_{r=1}^n A(X^r), \\
	\widehat{K} = \frac{1}{n}\sum_{r=1}^n K(X^r).
\end{align}
%\begin{align}
%	\widehat{\Gamma}_i := \frac{1}{n}\sum_{r=1}^n a_i(X^r)a_i(X^r)^\top, \\
%	\widehat{K} := \frac{1}{n}\sum_{r=1}^n K(X^r).
%\end{align}

%Let $\widehat{\Gamma}:=\text{diag}(\widehat{\Gamma}_i)$. 
The \emph{regularized score matching estimator} is a solution to the problem
\begin{align}
	\widehat{\theta}\in \underset{\theta}{\text{argmin}} \left\{\frac{1}{2}\theta^\top \widehat{\Gamma}\theta + \widehat{K}^\top \theta+\mathcal{R}(\theta)\right\}. \label{scorematch}
\end{align}

Here $\mathcal{R}$ is the group penalty
\begin{align}
	\mathcal{R}(\theta) = \sum_{i,j\in V} \Vert\theta_{ij}\Vert_2.
\end{align}
This norm induces sparsity in groups (i.e. edges/vertices). In high dimensions, regularizing the vertex parameters is necessary, as \eqref{scorematch} need not exist otherwise. Both the scoring rule and regularizer of \eqref{scorematch} are convex in $\theta$, so it is a convex program. In particular, observe that it can be equivalently represented as
\begin{align}
	&\min_{t,t_{ij}}\left\{ t + \lambda\sum_{ij} t_{ij}\right\} \label{socpform}\\
	s.t. \qquad  &t  \geq \frac{1}{2}\theta^\top\widehat{\Gamma}\theta+\widehat{K}^\top\theta, \notag \\
	&t_{ij} \geq \Vert \theta_{ij}\Vert \notag. 
\end{align}
\eqref{socpform} is a second-order cone program (SOCP) \citep{boyd2004convex}, as the quadratic constraint can be re-written as a conic constraint. If $\widehat{\Gamma}$ is not positive definite, particularly when $n>d$, \eqref{scorematch} may not be unique. This is typical for high dimensional problems. One can impose further assumptions to guarantee uniqueness. For example various assumptions have been described for the lasso (see an overview of these assumptions in \citep{tibshirani2013lasso}) , but we won't go into those details here.

\subsection{Gaussian Score Matching}

Consider the Gaussian density:
\begin{align}
	q(x) \propto \exp\left\{ - \frac{1}{2} x^\top \Omega x \right\},
\end{align}
for $\Omega\succ 0$. We have
\begin{align}
	\nabla \log q (x) &= -\Omega x, \\
	\nabla_i (\nabla_i \log q(x)) &= -\Omega_{ii}.
\end{align}
so the Hyv{\"a}rinen score is given by
\begin{align}
	h(x,\Omega) &= -\sum_i \Omega_{ii}+\frac{1}{2} x^\top \Omega^\top\Omega x \\ 
	&= \text{trace}\left(-\Omega + \frac{1}{2}\Omega^2 x x^\top \right).
\end{align}

%Note that the corresponding entropy is $\mathbb{E}_p[h(X,\Omega)]= -\frac{1}{2}\text{trace}\left( \Omega\right)$.

Let $\widehat{\Sigma}=\frac{1}{n}\sum_{r=1}^n X^r(X^r)^\top$. The  optimal regularized score estimator $\widehat{\Omega}$ is the solution to
\begin{align}
	&\min_{\Omega=\Omega^\top} \left\{ \text{trace}\left(\frac{1}{2}\Omega\widehat{\Sigma}\Omega-\Omega\right) + \lambda\Vert\Omega\Vert_1\right\}.\label{gaussreg}
\end{align}

In the notation of \eqref{scorematch}, we have $\theta=\text{vec}(\Omega)$, $\widehat{K}=\text{vec}(I_d)$ and  $\widehat{\Gamma}_i=\widehat{\Sigma}$ for each $i\in V$.
We do not impose a positive definite constraint on $\Omega$. Doing so would still result in a convex program, indeed it is a semidefinite program, but the resulting computation becomes more complicated and less scalable in practice. However, our theoretical results imply that $\widehat{\Omega}$ is positive definite with high probability. Indeed, denote $\Vert\widehat{\Omega}-\Omega^*\Vert_{sp}$ the spectral norm (maximum absolute value of eigenvalues) of the difference $\widehat{\Omega}-\Omega^*$. Since the spectral norm is dominated by the Frobenius norm (elementwise $L_2$ norm), the consistency result in the sequel implies consistency in spectral norm, and so the eigenvalues of $\widehat{\Omega}$ will be positive with probability approaching one, assuming the population precision matrix $\Omega^*$ has strictly positive eigenvalues. Furthermore, we note that our model selection guarantees still follow whether or not the estimator $\widehat{\Omega}$ is positive definite.

%Suppose the data are scaled to have unit variance, so that $\Sigma$ is the correlation matrix, and fix the diagonal $\Omega_{ii}=1$.  Thus after the $\frac{1}{2}$ term is removed we  get the expression
%\begin{align}
%&\min_\Omega\text{trace}\left((\Omega-\Sigma^{-1})\Sigma\Omega\right) + \lambda\Vert\Omega\Vert_1 \\
%	&= \min_\Omega \text{trace}\left((\Omega-\Sigma^{-1})\Sigma\Omega-I\right) + \lambda\Vert\Omega\Vert_1 \\
%	&= \min_\Omega \text{trace}\left((\Omega-\Sigma^{-1})\Sigma(\Omega-\Sigma^{-1})\right) + \lambda\Vert\Omega\Vert_1. \label{mbproc}
%\end{align}
%
%\citep{yuan2007model} show that when the symmetry constraints on $\Omega$ is relaxed, \eqref{mbproc} estimates precisely the parameters of the parallel LASSO procedure \citep{meinshausen2006high} for model selection in the graphical model. The objective can be split into $d$ LASSO problems which may be solved in parallel. Further, the score is very close to the quadratic approximation to the negative Gaussian log-likelihood, 
%\begin{align}
%\text{trace}\left((\Omega-\Sigma^{-1})\Sigma(\Omega^\top-\Sigma^{-1})\Sigma\right)
%\end{align}
%Note that in the high-dimensional setting no empirical version of $\Sigma^{-1}$ is readily available as the sample covariance matrix $\widehat{\Sigma}$ is not invertible when $d>n$. Thus, these derivations are illustrative and in practice one would use the expression in \eqref{gaussreg}.

\section{Main Results}
We suppose we are given i.i.d. data $X^1,\ldots,X^n \sim p^*$. $p^*$ need not belong to the pairwise exponential family being estimated, in which case we may think of our consistency results as being relative to the population quantity
%\begin{align}
%	\theta^* &:= \text{argmin}_{\theta:\theta_{ij}=\theta_{ji}} \frac{1}{2} \theta^\top \mathbb{E}_{p^*}[A(X)]\theta+\theta^\top\mathbb{E}_{p^*}[K(X)].
%\end{align}
\begin{align}
	\theta^* &:= \left\{\mathbb{E}_{p^*}[A(X)]\right\}^{-1}\mathbb{E}_{p^*}[K(X)] \\
	&= (\Gamma^*)^{-1}K^*. \label{popsol}
\end{align}
Define the maximum column sum of $\theta^*$ by 
\begin{align}
\kappa_{\theta,1}:=\max_{i\in V} \sum_{j\in V,u\leq m} (\theta^*)_{ij}^u,
\end{align}
and define the maximum degree as
\begin{align}
s := \max_{i\in V}\left| \{(i,j):(i,j)\in E\}\right|.
\end{align}
%We suppose $s,\kappa_{\theta,1}$ are bounded as $n$ increases.

\begin{assump} \label{eigenbound}
$\Gamma_i^*=\mathbb{E}_{p^*}[a_i(X)a_i(X)^\top]$ satisfies for each $i\in V$,
\begin{align}
	\infty>\bar{\epsilon}\geq\Lambda_{\max}(\Gamma_i^*)\geq\Lambda_{\min}(\Gamma_i^*) \geq \underline{\epsilon}>0.
\end{align}
Note that this also implies that the eigenvalues of $\Gamma^*$ are bounded as the rows of $\Gamma^*$ are non-trivial linear combinations of those of $\text{diag}(\Gamma_i^*)$, so the inverse in \eqref{popsol} exists and is unique.
\end{assump}
We also suppose $\theta^*$ is sparse, in the following sense:

\begin{assump}
$\theta^*$ belongs to the set
\begin{align}
	\tilde{\mathcal{P}} := \tilde{\mathcal{P}}(E)= \left\{\theta: \Vert\theta_{ij}\Vert_2=0,\text{  for } (i,j)\in E^c\right\}.
\end{align}
\end{assump}
For both parameter consistency and model selection we require the following tail conditions:
\begin{assump} \label{quasrconcentrate}
	For each $i,j,k\in V$ and $u\leq m$ and $t\leq \nu$, for some $c_1,c_2,\nu>0$,
	\begin{align}
		\mathbb{P}\left( \left|\widehat{K}_{ij}^u-(K^*)_{ij}^u\right|\geq t\right) &\leq \exp\left\{ - c_1nt^2 \right\} \\
		\mathbb{P}\left( \left|(\widehat{\Gamma}_{i})_{jk}^u-(\Gamma^*_{i})_{jk}^u\right|\geq t\right) &\leq \exp\left\{ - c_2nt^2 \right\}.
	\end{align}
\end{assump}
\subsection{Parameter Consistency}
We present results in terms of the (vector) $L_2$ norm. Note in particular that this result doesn't require any incoherence condition (though we do require for model selection consistency in the sequel).

For the parameter consistency results in particular, we require the following sub-Gaussian assumption:

\begin{assump}
For each $i\in V$ and $r=1,\ldots,n$, $a_i(X^r)$ is a sub-Gaussian random vector.
\end{assump}

\begin{thm}\label{parconsist}
Suppose the regularization parameter is chosen as
\begin{align}
	\lambda_n\asymp\sqrt{\frac{m\kappa_{1,\theta}^2\log(md)}{n}},
\end{align}
if the sample size satisfies
\begin{align}
	n=\Omega(md),
\end{align}
then any solution to regularized score matching satisfies
\begin{align}
	\Vert\widehat{\theta}-\theta^*\Vert_{2} &=O_p\left(\sqrt{\frac{(d+\left|E\right|)m\kappa_{1,\theta}^2\log(md)}{n}}\right).
\end{align}
\end{thm}

\begin{remark}
	Consider Gaussian score matching. Here $m=1$, so if $\kappa_{1,\theta}$ is bounded we have $\Vert\widehat{\theta}-\theta^*\Vert_2 =O_p\left(\sqrt{\frac{(d+\left|E\right|)\log(d)}{n}}\right)$. This rate is the same as the graphical lasso shown in \citep{rothman2008sparse}. Furthermore, here $\Gamma_i^*=\Sigma$, so our assumption \ref{eigenbound} amounts to bounds on the eigenvalues of $\Sigma$, which are the same as for sparse precision matrix MLE. The assumption that $\kappa_{1,\theta}$ is bounded here says that the  sums of the absolute value of rows of $\Omega^*$ are bounded, which is not necessary for the regularized MLE.
\end{remark}

\begin{remark}
We might reasonably expect $\kappa_{1,\theta}=O(sm)$, in which case $\Vert\widehat{\theta}-\theta^*\Vert_2 = O_p\left(\sqrt{\frac{(d+\left| E\right|)m^3s^2\log(md)}{n}}\right)$. In this setting the regularized MLE will have the rate 
\begin{align}
O_p\left(\sqrt{\frac{(d+\left| E\right|)m \log(md)}{n}}\right).
\end{align}
 (see results in \cite{janofsky2015exponential}, Appendix A).
\end{remark}

\subsection{Model Selection} \label{smms}

For model selection we require several additional conditions. Denote $\widehat{E}$ as the edge set learned from $\widehat{\theta}$:
\begin{align}
	\widehat{E}:=\left\{(i,j): \Vert\widehat{\theta}_{ij}\Vert_2=0\right\}.
\end{align}
Furthermore, define 
\begin{align}
	\kappa_\Gamma := \Vert (\Gamma^*)^{-1}\Vert_\infty, \\
	\kappa_\theta := \Vert\theta^*\Vert_{\max},\\
	\rho^* := \underset{(i,j)\in E}{\min} \Vert\theta_{ij}\Vert_{\max}.
\end{align}

Here $\Vert A\Vert_\infty = \max_j\sum_i \left| A_{ij}\right|$ is the matrix $\infty$ norm and $\Vert\cdot\Vert_{\max}$ the elementwise max norm. We require an \emph{incoherence condition}:
\begin{assump}
\begin{align}
\max_{(i,j)\in E^c} \Vert \Gamma_{ij,E}^*(\Gamma^*_{EE})^{-1}\Vert_2 \leq \frac{1-\tau}{\sqrt{d+E}}, && \text{for some } \tau\in (0,1].
\end{align}
where $\Vert A\Vert_2$ is the matrix operator norm. 
\end{assump}

In the following theorem we suppose $\kappa_\Gamma,\kappa_\theta,s$ are are bounded, while $\rho^*$ may change with the sample size.

\begin{thm}\label{modelselectthm}
	Suppose the regularization parameter $\lambda_n$ is chosen to be
	\begin{align}
		\lambda_n \asymp \sqrt{\frac{m\kappa_{1,\theta}^2 \log(dm)}{n}},
	\end{align}
	then if
	\begin{align}
		n&=\Omega(\max\{m\kappa_{1,\theta}^2\log(dm),m^2s^2\log(dm)\}), \\
		\frac{1}{\rho^*} &= o\left(\sqrt{\frac{\kappa_{1,\theta}^2\log(dm)}{n}}\right),
	\end{align}
	there exists a solution to the regularized score matching estimator $\widehat{\theta}$ with estimated edge set $\widehat{E}$ satisfying
\begin{align}
	\mathbb{P}(\widehat{E}=E)\rightarrow 1.
\end{align}
\end{thm}

\begin{remark}
	Assuming $m,s,\kappa_{1,\theta}$ are bounded, this implies the dimension may grow nearly exponentially with the sample size:
\begin{align}
		d=o(e^n),
\end{align}
with the probability of model selection consistency still aproaching one.
\end{remark}
\begin{remark}[Gaussian score matching]

When $m=1$, the sample complexity matches that for structure learning of the precision matrix using the log-det divergence, in \citep{ravikumar2011high}. Thus Gaussian score matching in particular  benefits from identical model selection guarantees as the graphical lasso algorithm. However it should be noted that the assumptions are slightly different. In particular the graphical lasso requires an irrepresentable condition on $\Sigma\otimes\Sigma$, while our method involves an irrepresentable condition for $\Sigma\otimes I_d$.
\end{remark}
\subsection{Model Selection for the Nonparametric Pairwise Model}

In this section we consider model selection for the nonparametric pairwise model. We suppose the log of the true density $p^*$ belongs to $W_2^r$, the Sobolev space of order $r$. This implies, along with the pairwise assumption, that $\log p^*$ has the infinite expansion
\begin{align}
	\log p^* &\propto \exp\left\{ \sum_{i,j\in V,i\leq j}\sum_{k,l=1}^\infty (\theta^*)_{ij}^{kl}\phi_{kl}(x_i,x_j)+\sum_{i\in V}\sum_{k=1}^\infty (\theta^*)_{i}^k\phi_k(x_i)\right\}, \label{infexpans}
\end{align}
where here $\{\phi_k,\phi_{kl}\}$ is a basis over $[0,1]^2$. For an expansion in  $W_2^r$, we have that the coefficients decay at the following rates:
\begin{align}
	&\sum_k \theta_i^k k^{2r} <\infty, & \text{ for all } i\in V,\\
	&\sum_{k,l} \theta_{ij}^{kl} k^{2r_i}l^{2r_j}<\infty, & \text{ for all } (i,j)\in E,& \qquad r_i+r_j=r.
\end{align}

 For our results we assume $\{\phi_k\}$ is the orthonormal Legendre basis on $[0,1]$, and $\{\phi_{kl}\}$ is the tensor product basis $\phi_{kl}(x_i,x_j)=\phi_k(x_i)\cdot \phi_l(x_j)$. This is because the supporting lemmas are particular to the Legendre basis,  but in practice one is not limited to a particular basis. Now, consider forming a density by truncating \eqref{infexpans} after $m_1$ terms for the univariate expansions, and $m_2$ for bivariate:
\begin{align}
	\log p_{\theta} &\propto \exp\left\{ \sum_{i,j \in V,i\leq j}\sum_{k,l=1}^{m_2} \theta_{ij}^{kl}\phi_{kl}(x_i,x_j)+\sum_{i\in V}\sum_{k=1}^{m_1}\theta_{i}^k\phi_k(x_i)\right\}.
\end{align}

Observe that this is a finite-dimensional exponential family. Furthermore, the normalizing constant for this family will generally be intractable, requiring a $d$-fold integral. We choose our density estimate to be $p_{\widehat{\theta}}$, where $\widehat{\theta}$ is a solution to the score matching estimator \eqref{scorematch} for this family. Furthermore, we let the number of sufficient statistics $m_1,m_2$ grow with the sample size $n$ to balance the bias from truncation with the estimation error. We denote $E$ to be the support of $p^*$:
\begin{align}
	E:= \left\{ (i,j): \Vert\theta_{ij}^*\Vert=0\right\},
\end{align}

Now, decompose the vector $\theta^*$ into the included terms and truncated terms, $\theta^*=((\bar{\theta}^*)^\top,(\theta^*_T)^\top)^\top$ and corresponding sufficient statistics $\phi=((\bar{\phi})^\top,(\phi_T)^\top)$. Denote $(a_T)_i(x) := \frac{\partial}{\partial x_i} (\phi_T){\cdot,i}$, $A_T(x)= \text{diag}((a_T)_i(x)(a_T)_i(x)^\top)$, and $\Gamma_T^* =\mathbb{E}_p [A_T(X)]$. Applying the results in Section \ref{smef}, we have the linear relation
\begin{align}
	K^* = -\Gamma^*_T\theta^*_T - \Gamma^* \bar{\theta}^*.
\end{align}

In the following theorem we assume $\kappa_T := \Vert\Gamma^*_T\Vert_{\max}$ is bounded, and $\kappa_{1,\theta}=O(m_2^2)$, in addition to the assumptions for the parametric setting stated in Section \ref{smms}, with the exception of Assumption \ref{quasrconcentrate}. Since the number of statistics grows to infinity in the nonparametric case, we need more accurate accounting of the constant terms in the concentration inequality. In lieu of the concentration assumption, we have the following assumption on the boundedness of the marginals of $p$.

\begin{assump} \label{assump4}
		 For each $i,j\in V$,
	\begin{align}
		\underline{\epsilon} & \leq p_{ij}(x_i,x_j) \leq \bar{\epsilon},
		\end{align}
		for absolute constants $\underline{\epsilon}>0,\bar{\epsilon}<\infty$.
\end{assump}

This assumption is mild for density estimation as it only requires bounds on the bivariate marginals rather than the full distribution.

\begin{thm}\label{msnonp}
	Suppose that the truncation parameters and regularization parameter are chosen to be
	\begin{align}
		m_2\asymp n^{\frac{1}{2r+13}} \\
		m_1\asymp n^{\frac{1}{2r+13}} \\
		\lambda_n\asymp\sqrt{\frac{\log nd}{n^{\frac{2r-1}{2r+13}}}} \
	\end{align}
	and the dimension $d$ and $\rho^*$ satisfy
	\begin{align}
		d= o\left(e^{n^{\frac{2r-1}{2r+13}}}\right) \\
		\frac{1}{\rho^*} = o\left(\sqrt{\frac{\log nd}{n^{\frac{2r+1}{2r+13}}}}\right),
	\end{align}
	then there exists a solution $\widehat{\theta}$ such that the edge set $\widehat{E}$ satisfies
	\begin{align}
		\mathbb{P}(\widehat{E}=E)\rightarrow 1.
	\end{align}
\end{thm}
\begin{remark}
	If $r=2$, and $s$ grows as a constant, we may have
	\begin{align}
		d=o\left(e^{n^{3/17}}\right),
	\end{align}
	and still ensure model selection consistency. In \cite{janofsky2015exponential} it was shown that the sample complexity for model selection in the nonparametric pairwise model the regularized exponential series MLE using Legendre polynomials is $d=o\left(e^{n^{3/7}}\right)$, though this estimator can't be computed exactly. The optimal choice of regularization and truncation parameters is much different for these two methods. This is a consequence of different estimation errors. In our supporting lemmas we require convergence of the statistic $\widehat{\Gamma}$ to its expectation. In the appendix we show that applying Hoeffding's inequality and a union bound,
\begin{align}
	\kappa_{1,\theta}\Vert\widehat{\Gamma}-\Gamma\Vert_{\max} = O_p\left(\sqrt{\frac{m_2^{12} \log nd}{n}}\right).
\end{align}

For the regularized MLE, we needed convergence of the sufficient statistics $\widehat{\mu}$, which converges at a much faster rate of $O_p\left(\sqrt{\frac{m_2^4 \log nd}{n}}\right)$. Our results agree with intuition, that the score matching statistics, derived from the derivatives of the log-density, should be harder to estimate than the sufficient statistics.

Also, it should be noted that the assumptions underlying the two results are different. The MLE involves conditions on the covariance of the sufficient statistics $\text{cov}_{p^*}[\phi(X)]$, while the score matching estimator requires conditions on $\Gamma^*=\mathbb{E}_{p^*}[A(X)]$. An interesting stream of future work would be to better understand the relationship between these two approaches and their assumptions.
\end{remark}
\section{Algorithms}
In this section we consider algorithms for solving \eqref{scorematch}. In our experiments we denote our method QUASR, for \textbf{Qua}dratic \textbf{S}coring and \textbf{R}egularization. There are variety of generic approaches to solving problems which may be cast as the sum of a smooth convex function plus a sparsity-inducing norm  \citep{bach2011convex}, as well as generic solvers for solving second-order cone programs. Here we will propose two novel algorithms which exploit the unique structure of the problem at hand. First we will consider an ADMM algorithm; for a detailed exposition of this approach, see \citep{boyd2011distributed}. In section \ref{cwd}, we consider a coordinate-wise descent algorithm for Gaussian score matching \citep{friedman2007pathwise}.
\subsection{Consensus ADMM} \label{cadmm}

The idea behind ADMM is that the problem \eqref{scorematch} can be equivalently written as
\begin{align}
	\underset{\theta,z}{\text{min}} \left\{\frac{1}{2}\sum_{i\in V} \left(\theta_{\cdot,i}^\top \widehat{\Gamma}_i \theta_{\cdot,i} +\theta_{\cdot,i}^\top\widehat{K}_{\cdot,i}\right)+ \lambda\sum_{i,j\in V,i\leq j}\Vert z_{ij}\Vert_2\right\},
\end{align}
subject to the constraint that $\theta_{ij}=\theta_{ji}=z_{ij}$. The scaled augmented Lagrangian for this problem is given by
\begin{align}
	L\left(\theta,y,z\right) &= \frac{1}{2}\sum_{i\in V}\left( \theta_{\cdot,i}^\top \widehat{\Gamma}_i \theta_{\cdot,i} +\theta_{\cdot,i}^\top\widehat{K}_{\cdot,i}\right) \\
	&\quad+\sum_{i,j\in V: i\leq j}\bigg( \Vert z_{ij}\Vert_2+y_{ij}^\top(\theta_{ij}-z_{ij})+y_{ji}^\top(\theta_{ji}-z_{ij}) \\
	&\qquad\qquad\qquad+\frac{\rho}{2}\Vert \theta_{ij}-z_{ij}\Vert^2+\Vert \theta_{ji}-z_{ij}\Vert^2\bigg),
\end{align}
here $\{y\}:=\left\{y_{ij},y_{ji}\right\}$ are dual variables, and $\rho$ is a penalty parameter which we choose to be 1 for simplicity. The idea behind ADMM is to iteratively optimize $L$ over the $\theta,y,z$ variables in turn. In the first step, since $\theta_{ij}=\theta_{ji}$ is included as a constraint and may be considered separately, $L$ as a function of $\theta$ decouples into $d$ independent quadratic programs, one for each "column" of $\theta$, which may be solved in parallel. In the second step, $z_{ij}$ pools the estimates $\theta_{ij}$ and $\theta_{ji}$ from the previous step, and applies a group shrinkage operator. The third step is a simple update of the dual variables.

\begin{figure}
\begin{center}
\framebox{\begin{minipage}[t]{0.95\columnwidth}
\begin{enumerate}
	\item Initialize $\theta^{(0)}$, $z^{(0)}$, $y^{(0)}$, and choose $\rho=1$;
	\item For $t=1,\ldots,$ until convergence:
		\begin{enumerate}
			\item Update $\theta$ for $i\in V$: 
			\begin{align}
	\theta_{\cdot,i}^{(t)} = \left(\widehat{\Gamma}_i+\rho I_d\right)^{-1}\left(-\widehat{K}_{\cdot,i}-y_{\cdot,i}^{(t-1)}+\rho z_{\cdot,i}^{(t-1)}\right),
\end{align}
	
			\item Update $z$ for $i,j\in V$, $i\leq j$:
			\begin{align}
				z_{ij}^{(t)} = \tilde{S}\left(\frac{1}{2}\left(\theta_{ij}^{(t)}+\theta_{ji}^{(t)}+y_{ij}^{(t-1)}/\rho+y_{ji}^{(t-1)}/\rho\right),\lambda/\rho\right),
			\end{align}
			where $\tilde{S}(x,\lambda):= \left(1-\frac{\lambda}{\Vert x\Vert_2}\right)_+ x$.
			\item Update $y$ for $i,j\in V$, $i\leq j$:
			\begin{align}
				y_{ij}^{(t)} = y_{ij}^{(t-1)} + \rho\left(x_{ij}^{(t)}-z_{ij}^{(t)}\right),\\
				y_{ji}^{(t)} = y_{ji}^{(t-1)} + \rho\left(x_{ji}^{(t)}-z_{ij}^{(t)}\right).
			\end{align}
		\end{enumerate}		
\end{enumerate}
\end{minipage}}
\caption{QUASR Consensus ADMM}
\end{center}
\end{figure}
Due to parallel updating of $\theta$ in step (a) and subsequent averaging in step (b), this is known as \emph{consensus ADMM}. At convergence, the constraints $\theta_{ij}=\theta_{ji}=z_{ij}$ are binding. In practice, we stop when the average change in parameters is small:
\begin{align}
	\sum_{i,j\in V} \Vert\theta_{ij}^{(t)}-\theta_{ij}^{(t-1)}\Vert_1\big/\sum_{i,j\in V} \Vert\theta_{ij}^{(t)}\Vert_1 < 10^{-4}.
\end{align}

In addition to parallelizing the update (a), other speedups are possible. For example, we may compute the eigenvalues $\Lambda$ and eigenvectors $Q$ of $\widehat{\Gamma}_i$, which may be computed directly from the data matrix $\left[a_i(X^1),\ldots a_i(X^n)\right]^\top$ using the singular value decomposition ($Q$ being the right singular vectors, and $\sqrt{n}\Lambda$ being the squared singular values of the data matrix). We may then cache the matrix $Q(\text{diag}(\Lambda+\rho)^{-1})Q^\top$, which is equivalent to $(\widehat{\Gamma}_i+\rho I_d)^{-1}$ up to numerical error. This can be computed for each $i\in V$, also in parallel, and only needs to be computed once (even if estimating over a sequence of $\lambda s$). When optimizing over a path of truncation parameters $m_1,m_2$, one may utilize block matrix inversion formulas and the Woodbury matrix identity to avoid computing the inverse from scratch each time. In particular, let $\widehat{\Gamma}_i$ be the current matrix of statistics, and $\widehat{\Gamma}_i^{new}$ be the statistic with a higher degree of basis expansion. Then $\widehat{\Gamma}_i^{new}$ has the form for some $\widehat{b},\widehat{C}$,
\begin{align}
	\widehat{\Gamma}_i^{new} = \left(\begin{array}{cc}
		\widehat{\Gamma}_i & \widehat{b} \\
		\widehat{b}^\top & \widehat{C}
	\end{array}
	 \right).
\end{align}
The inverse takes the form
\begin{align}
	(\widehat{\Gamma}_i^{new}+\rho I)^{-1} = \left(\begin{array}{cc}
		L &  \\
		-\left(\widehat{C}+\rho I\right)^{-1}\widehat{b}^\top L & \left(\widehat{C}+\rho I -\widehat{b}^\top \left(\widehat{\Gamma}_i+\rho I\right)^{-1}\widehat{b}\right)^{-1}
	\end{array}
	 \right),
\end{align}
where
\begin{align}
	L&:=\left(\widehat{\Gamma}_i+\rho I - \widehat{b}^\top \left(\widehat{C}+\rho I\right)^{-1} \widehat{b}\right)^{-1} \\
	&= \left(\widehat{\Gamma}_i+\rho I\right)^{-1} - \left(\widehat{\Gamma}_i+\rho I\right)^{-1}\widehat{b}^\top\left(\widehat{C}+\rho I -\widehat{b}^\top \left(\widehat{\Gamma}_i+\rho I\right)^{-1}\widehat{b}\right)^{-1}\widehat{b}\left(\widehat{\Gamma}_i+\rho I\right)^{-1}.
\end{align}

If the dimension of $\widehat{C}$ is small relative to that of $\widehat{\Gamma}_i$, $\left(\widehat{\Gamma}_i^{new}+\rho I\right)^{-1}$ can be computed quickly using the cached $\left(\widehat{\Gamma}_i + \rho I\right)^{-1}$, without the need for any additional large matrix inversions.
% Thus, each iterate only requires simple operations, including $d$ parallel matrix solves and $d(d+1)/2$ parallel group thresholding operations. Also, when estimating over a sequence of $\lambda$s, we may employ warm starts, where the algorithm is initialized with the output $\theta,z,y$ from the previous run in the sequence.

\subsection{Coordinate-wise Descent} \label{cwd}
In this section we consider a coordinate-wise descent algorithm for the Gaussian score matching problem \eqref{gaussreg}. Coordinate-wise descent algorithms are known to be state-of-the-art for many statistical problems such as the lasso and group lasso \citep{friedman2007pathwise} and glasso for sparse Gaussian MLE \citep{friedman2008sparse}. Regularized score matching in the Gaussian case admits a particularly simple coordinate update.
Consider the stationary condition for $\Omega$ in \eqref{gaussreg}: 
\begin{align}
	\frac{1}{2}\left(\Omega\widehat{\Sigma}+\widehat{\Sigma}\Omega\right)-I_d + \widehat{Z}=0,
\end{align}
where $\widehat{Z}$ is an element of the subdifferential $\partial\Vert\Omega\Vert_1$:
\begin{align}
	\widehat{Z}_{ij} \in \begin{cases}
		\{\theta_{ij}: \Vert\theta_{ij}\Vert_2\leq 1\}, & \text{  if } \Vert\theta_{ij}\Vert=0; \\
		\frac{\theta_{ij}}{\Vert\theta_{ij}\Vert_2}, & \text{  if } \Vert\theta_{ij}\Vert\not=0.
	\end{cases}
\end{align}
in particular, the stationary condition for a particular $\Omega_{ij}$ is
\begin{align}
	\frac{1}{2}\left(\Omega_{\cdot,i}^\top\widehat{\Sigma}_{\cdot,j}+\widehat{\Sigma}_{\cdot,i}^\top\Omega_{\cdot,j}\right)-1\{i=j\}+\widehat{Z}_{ij}=0.\label{stationary}
\end{align}
%\begin{align}
%	&((\widehat{\Gamma}_i)_{jj}+(\widehat{\Gamma}_j)_{ii})\theta_{ij} + ((\widehat{\Gamma}_i)_{j,\backslash j})\theta_{i,\backslash j}+(\widehat{\Gamma}_j)_{i,\backslash j})\theta_{j \backslash i} \\
%	& +2\widehat{K}_{ij}+ \lambda\widehat{Z}_{ij}=0.
%\end{align}
%
%where $\widehat{Z}_{ij}\in\partial \mathcal{R}(\theta)_{ij}$. After some algebra, we get the update
%
%\begin{align}
%	\theta_{ij}\leftarrow S\left(-\frac{(\widehat{\Gamma}_i)_{j,\backslash j}\theta_{i,\backslash j} +(\widehat{\Gamma}_j)_{i,\backslash i}\theta_{j \backslash i}+2\widehat{K}_{ij}}{(\widehat{\Gamma}_i)_{jj}+(\widehat{\Gamma}_j)_{ii}},\lambda\right),
%\end{align}
%where $S(x,\lambda):=\text{max}\{\left| x\right|-\lambda,0\}\text{sign}(x)$.
Consider updating $\Omega_{ij}$ using equation \eqref{stationary}, solving for $\Omega_{ij}$ and holding the other elements of $\Omega$ fixed. After some manipulation, we get a fixed point for $\Omega_{ij}$ is given by \eqref{fixed}. We cycle through the entries of $\Omega$, applying this update, and repeat until convergence.

\begin{figure}
\begin{center}
\framebox{\begin{minipage}[t]{0.95\columnwidth}

\begin{enumerate}
	\item Initialize $\widehat{\Omega}=I_d$;
	\item For i=1,2,\ldots,d,1,2,\ldots, until convergence:
	\begin{enumerate}
	\item for j=i,\ldots,d:
			\begin{align}
	\widehat{\Omega}_{ij} \leftarrow S\left(-\frac{\left(\widehat{\Omega}_{\backslash j,i}^\top \widehat{\Sigma}_{\backslash j,j} + \widehat{\Omega}_{\backslash i,j}^\top \widehat{\Sigma}_{\backslash i,i} - 2\cdot 1\{j=i\} \right)}{\widehat{\Sigma}_{ii}+\widehat{\Sigma}_{jj}},\lambda\right) ,\label{fixed}
\end{align}
	and set $\widehat{\Omega}_{ji}=\widehat{\Omega}_{ij}$.
	\end{enumerate}
\end{enumerate}
\end{minipage}}
\caption{Gaussian QUASR Coordinate-wise descent}
\end{center}
\end{figure}
Here $S(x,\lambda)$ is the soft thresholding function $S(x,\lambda):=\text{max}\{\left| x\right|-\lambda,0\}\text{sign}(x)$, and $\backslash i := \{1,\ldots,i-1,i+1,\ldots,d\}$. Each update only requires two sparse inner products and a soft thresholding operation. As such, in our experiments this algorithm converges very quickly, sometimes much faster than glasso for the same set of data.

\subsection{Choosing Tuning Parameters}
As of yet we have not discussed how to practically choose the regularization parameter $\lambda$ and for nonparametric score matching, the truncation parameters $m_1,m_2$. We suppose the existence of a held-out tuning set; in the absence, one may use cross-validation. If the likelihood is available, for example if fitting Gaussian score matching, or for a fixed graph which is a tree, we minimize the negative log-likelihood risk in the held out set. In the absence of the likelihood, we choose the tuning parameters to minimize the Hyv\"arinen score of the held out set. For a discussion on using scoring rules as a replacement for the likelihood in model selection and using \emph{score differences} as surrogates for Bayes factors, see \citep{dawid2014theory}.

To save on computation, we use the idea of \emph{warm starts} which we detail in the sequel.
First, observe that the first-order necessary conditions for regularized score matching are:
\begin{align}
	 \widehat{\Gamma}\widehat{\theta}+\widehat{K}+\widehat{Z}=0,
\end{align}
where $\widehat{Z}$ denotes the sub gradient of the regularizer $\mathcal{R}$, at $\widehat{\theta}$, which is

\begin{equation}
	\widehat{Z}_{ij} = \begin{cases}
		\{x:\Vert x\Vert\leq 1\}, & \Vert\theta_{ij}\Vert = 0, \\
		\frac{\theta_{ij}}{\Vert\theta_{ij}\Vert}, & o/w. 	
	\end{cases}
\end{equation}
so $\widehat{\theta}=0$ when
\begin{equation}
	\lambda \geq \max_{ij}\Vert\widehat{K}_{ij}\Vert.
\end{equation}
This allows us to choose an upper bound $\lambda_{start}$ such that the solution will be the zero vector.

The idea behind warm starting is the following: we begin with estimating $\widehat{\theta}_{\lambda_{start}}=0$. Then we fit our model on a path of $\lambda$ decreasing from $\lambda_{start}$, initializing each new problem with the previous solution $\widehat{\theta}_{\lambda}$. The solution path for the regularized MLE is smooth as a function of $\lambda$, suggesting nearby choices of $\lambda$ will provide values of $\widehat{\theta}$ which are close to one another.
 
We can also incorporate warm-starting in choosing $m_1,m_2$. For a given $\lambda$, we first estimate the model for first-order polynomials, corresponding to $m_1=m_2=1$. We then increment the truncation parameters by increasing the degree of the polynomial of he sufficient statistics. We augment the previous parameter estimate vector with zeros in the place of the added parameters, and warm start \emph{ISTA} from this vector. See Section \ref{cadmm} for other computation savings when augmenting the sufficient statistics when choosing $m_1,m_2$.

\section{Experiments}

\subsection{Gaussian Score Matching}

We begin by studying Gaussian score matching, and comparing to the regularized Gaussian MLE, using the glasso package in $\tt{R}$ \citep{friedman2008sparse}. We consider experiments with two graph structures: in the first, a tree is generated randomly; this has $d-1$ edges. In the second, a graph is generated where an edge occurs between node $i$ and $j$ with probability $0.1$, denoted the Erd\"os-Renyi graph. This graph has expected number of edges $0.05\cdot d(d-1)$. The data is scaled to have unit variance and mean zero.

\subsubsection{Regularization Paths}
Figures \ref{regpath1} and \ref{regpath2} display regularization paths for one run of these simulations, where $d=100$ and $n$ is either 100 or 500. Relevant variables are plotted in black. For $n=500$, it appears that the score matching estimator does a better job screening out irrelevant variables for both graph types. For $n=100$ they perform similarly. The score matching estimator tends to produce nonzero parameter estimates which are larger in magnitude than the regularized MLE, which is more pronounced for $n<d$.

\begin{figure}
\centering
\includegraphics[scale=.5]{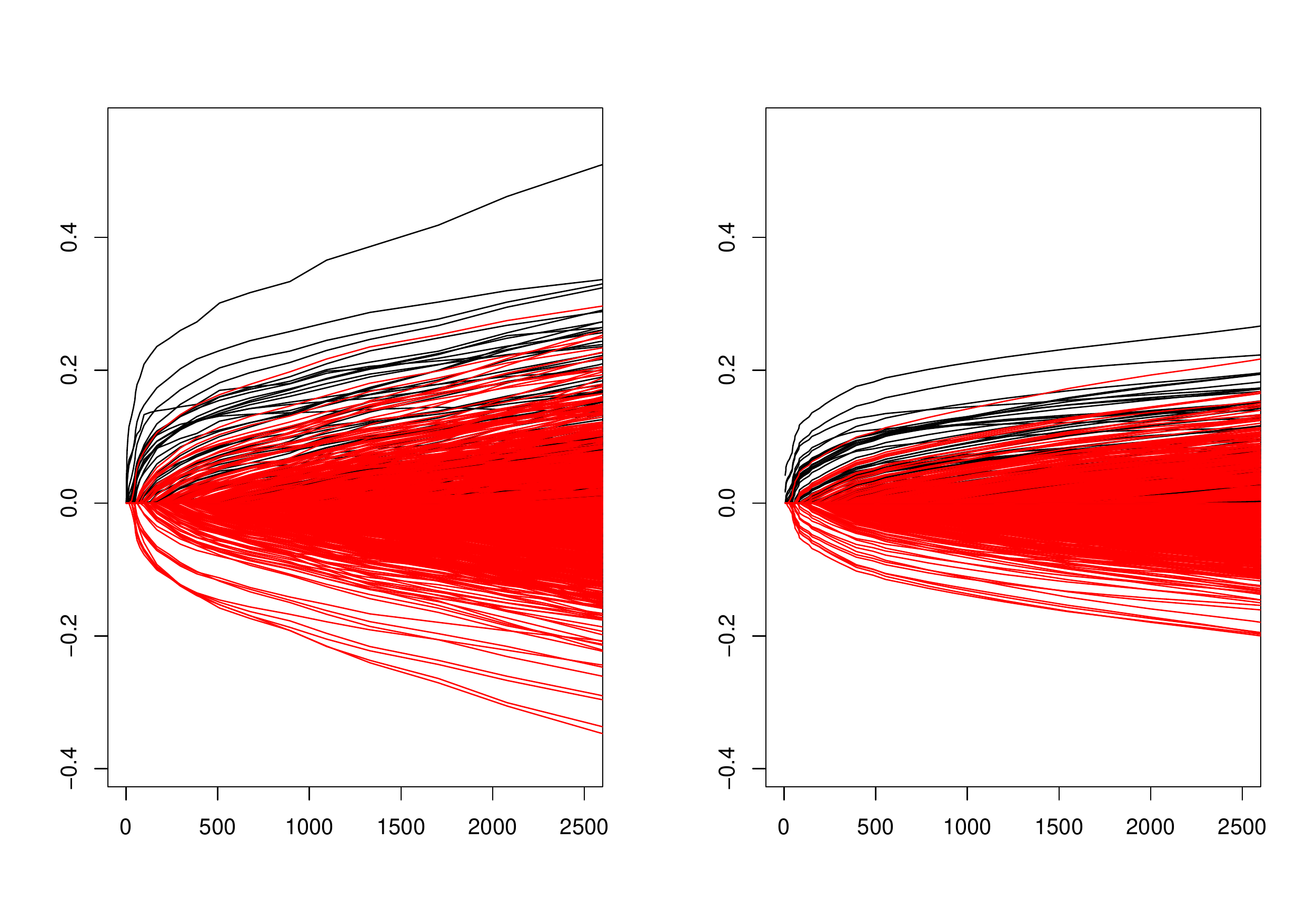}
\includegraphics[scale=.5]{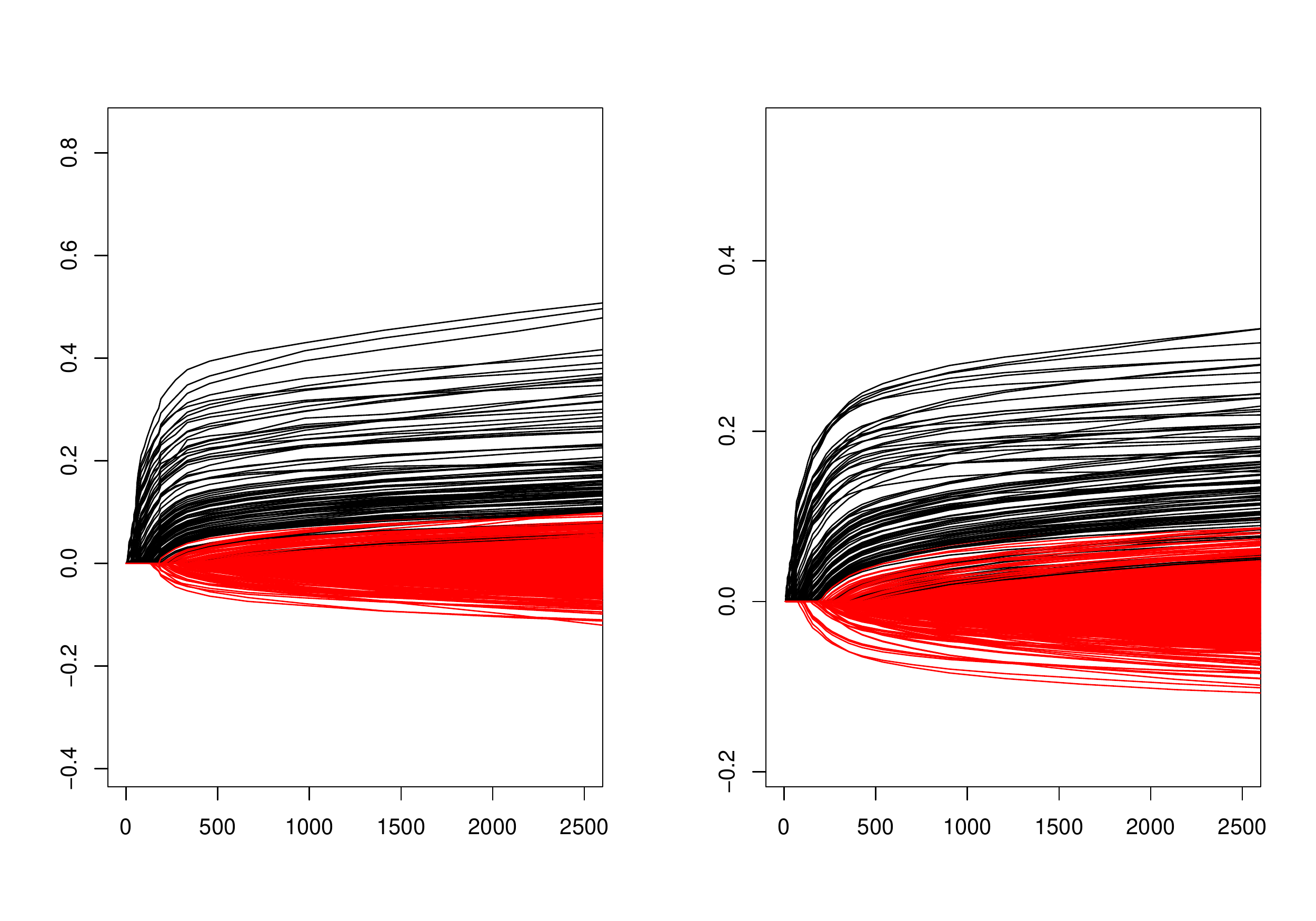}
\caption{Regularization path, tree graph. d=100. Top: n=100. Bottom: n=500.} \label{regpath1}
\end{figure}

\begin{figure}
\centering
\includegraphics[scale=.5]{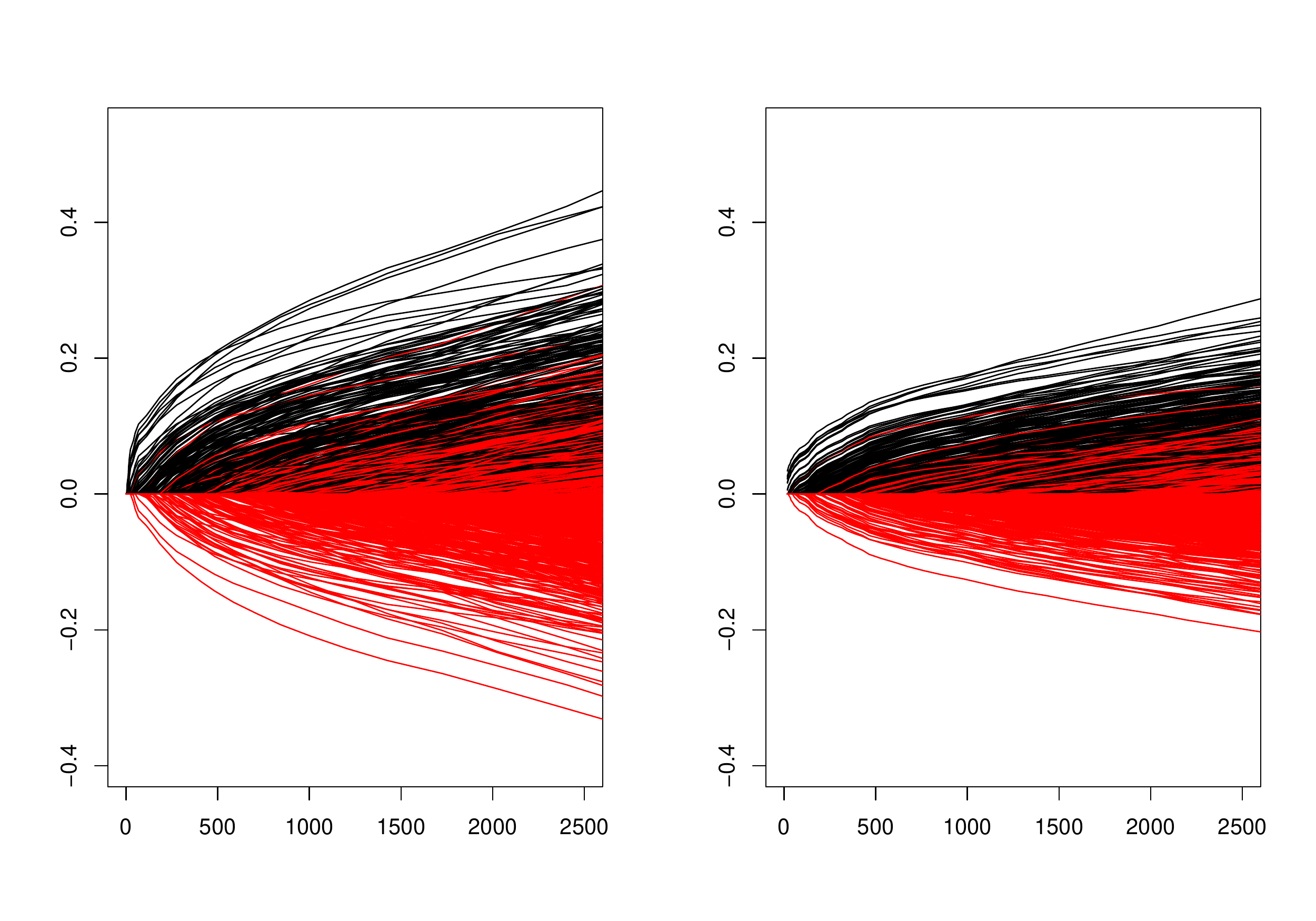}
\includegraphics[scale=.5]{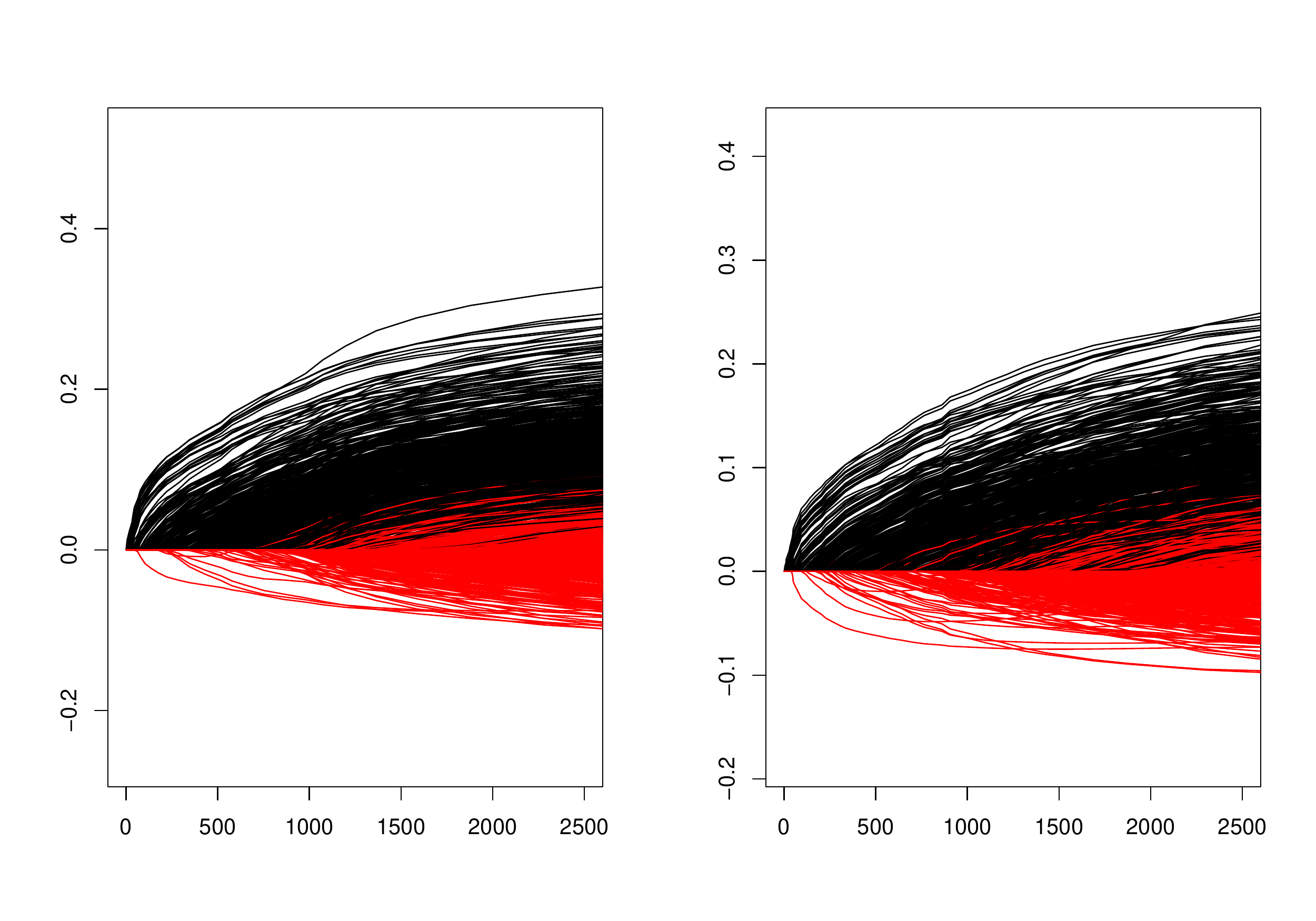}
\caption{Regularization path, Erdos-Renyi graph graph. d=100. Top: n=100. Bottom: n=500.} \label{regpath2}
\end{figure}

\subsubsection{Risk Paths}
Figures \ref{riskpath1} and \ref{riskpath2} show risk paths under the two graph structures; figure two has $d=150$, with 149 included edges; figure four has $d=100$, with 499 edges included. We choose $n=100$, and calculate the negative log likelihood risk using a held-out dataset of size $n$. The plotted curves are an average of 25 simulations from the same distribution.  In figure \ref{riskpath1}, we see the score matching estimator selects a sparser graph than the regularized MLE; furthermore, the score matching produces an estimator with smaller held-out risk. For the Erd\"os-Renyi simulation, the score matching estimator also selects a sparser graph, though it has risk slightly worse than the MLE.
\begin{figure}
\centering
\includegraphics[scale=.5]{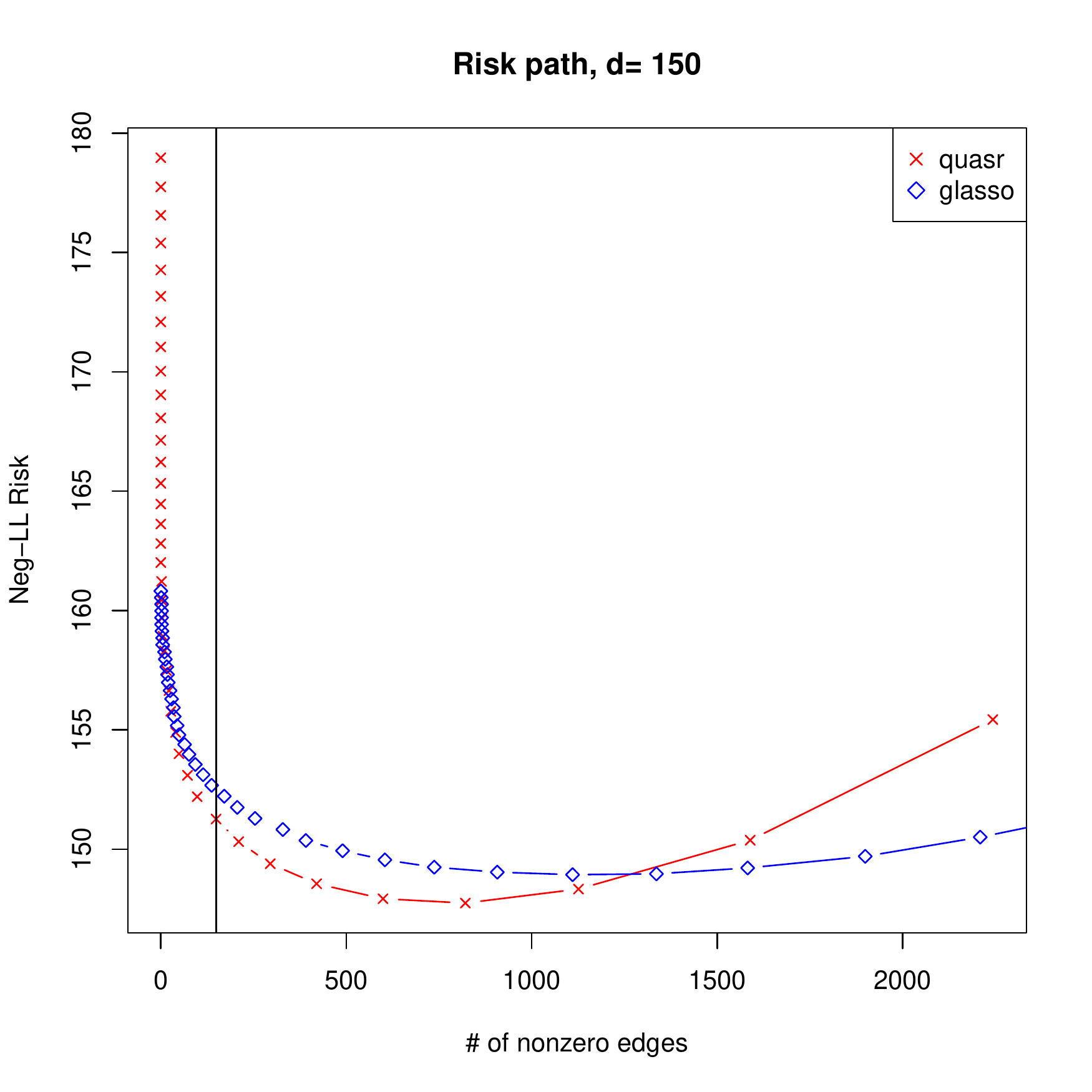}
\caption{Risk path, tree graph} \label{riskpath1}
\end{figure}

\begin{figure}
\centering
\includegraphics[scale=.5]{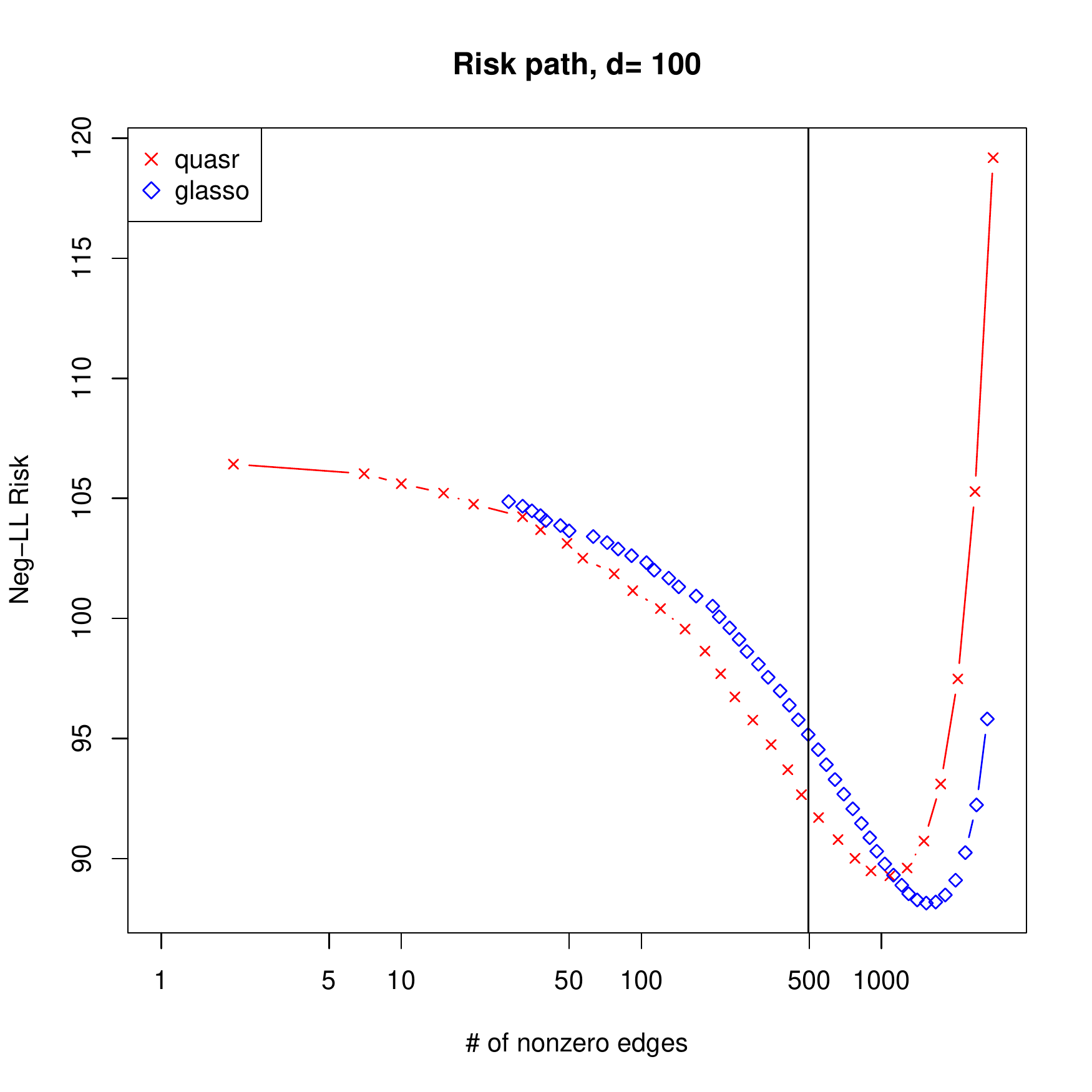}
\caption{Risk path, ER graph} \label{riskpath2}
\end{figure}
These findings are also validated in Tables \ref{table1} and \ref{table2}, varying $d$. For the tree graph, score matching dominates in risk for all values of $d$. Even for the Erd\"os-Renyi graph, score matching outperforms the regularized MLE when $d=30$.  Standard errors of 25 repetitions are in parentheses.
\begin{table}[ht]
\centering
\begin{tabular}{lll}
  \hline
 & quasr & glasso  \\
  \hline
d= 30 & 28.082 (0.341) & 28.189 (0.319)  \\
  d= 75 & 72.785 (0.472) & 73.202 (0.432)  \\
  d= 120 & 116.771 (0.578) & 117.631 (0.542) \\
  d= 150 & 147.26 (0.619) & 148.381 (0.583)  \\
   \hline
\end{tabular}
\caption{Held-out NLL error, tree graph} \label{table1}
\end{table}
\begin{table}[ht]
\centering
\begin{tabular}{lll}
  \hline
 & quasr & glasso  \\
  \hline
d= 30 & 26.157 (0.428) & 26.244 (0.386)  \\
  d= 50 & 44.351 (0.688) & 44.262 (0.718)  \\
  d= 100 & 91.383 (1.009) & 90.351 (1.077)  \\
  d= 150 & 139.03 (1.15) & 136.793 (1.336)  \\
   \hline
\end{tabular}
\caption{Held-out NLL error, ER graph} \label{table2}
\end{table}

\subsubsection{ROC Curves and Edge Selection}

Figures \ref{roc1} and \ref{roc2} show ROC curves under the same simulation setup in the previous section. The plotted points represent the graph selected from the minimal held-out risk in each of the 25 repetitions.
The two estimators display very similar ROC curves, and the score matching estimator tends to prefer higher sensitivity for lower specificity, when selecting using held-out data.

\begin{figure}
\centering
\includegraphics[scale=.5]{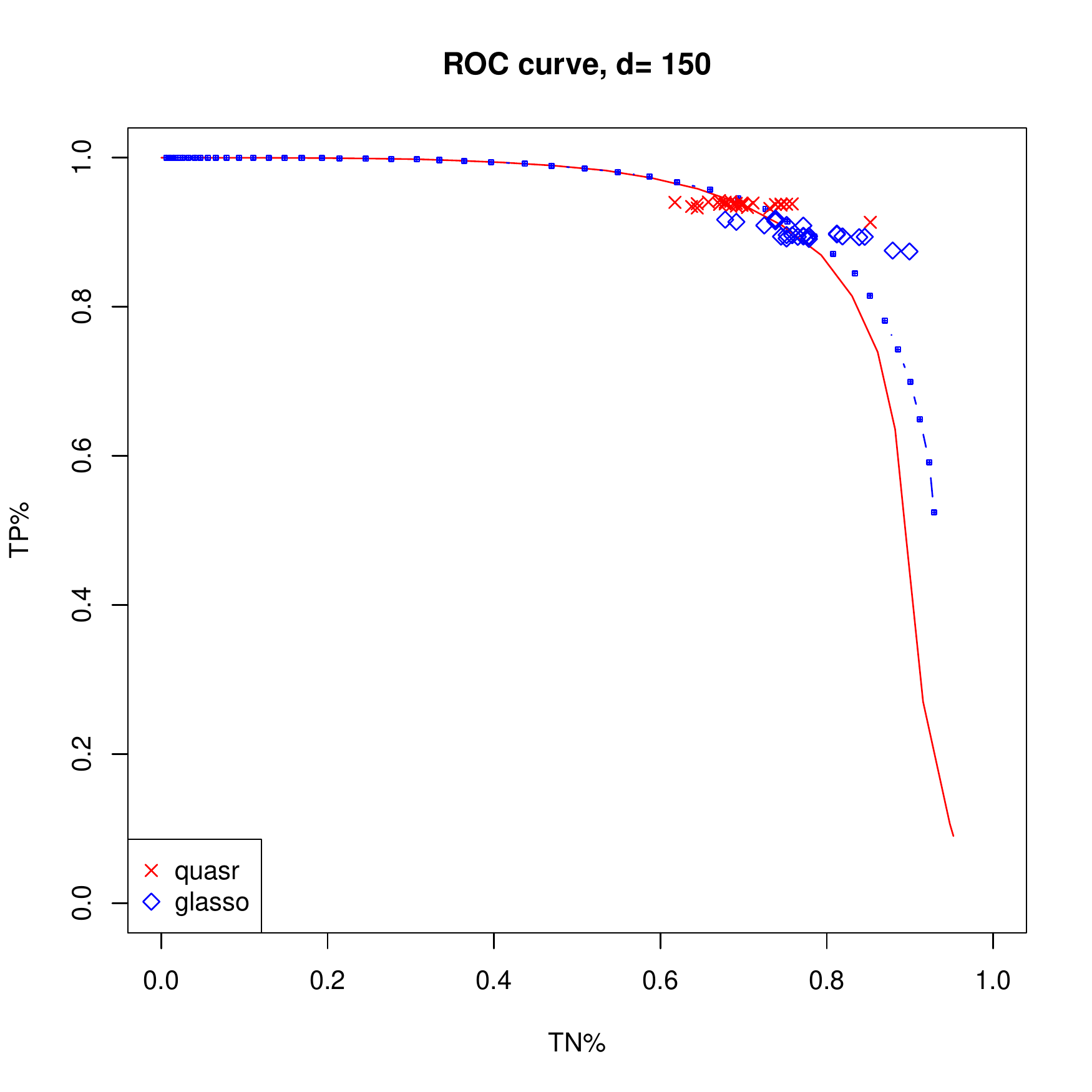}
\caption{ROC curve, tree graph. Solid line: Gaussian QUASR; dotted line: glasso.} \label{roc1}
\end{figure}

\begin{figure}
\centering
\includegraphics[scale=.5]{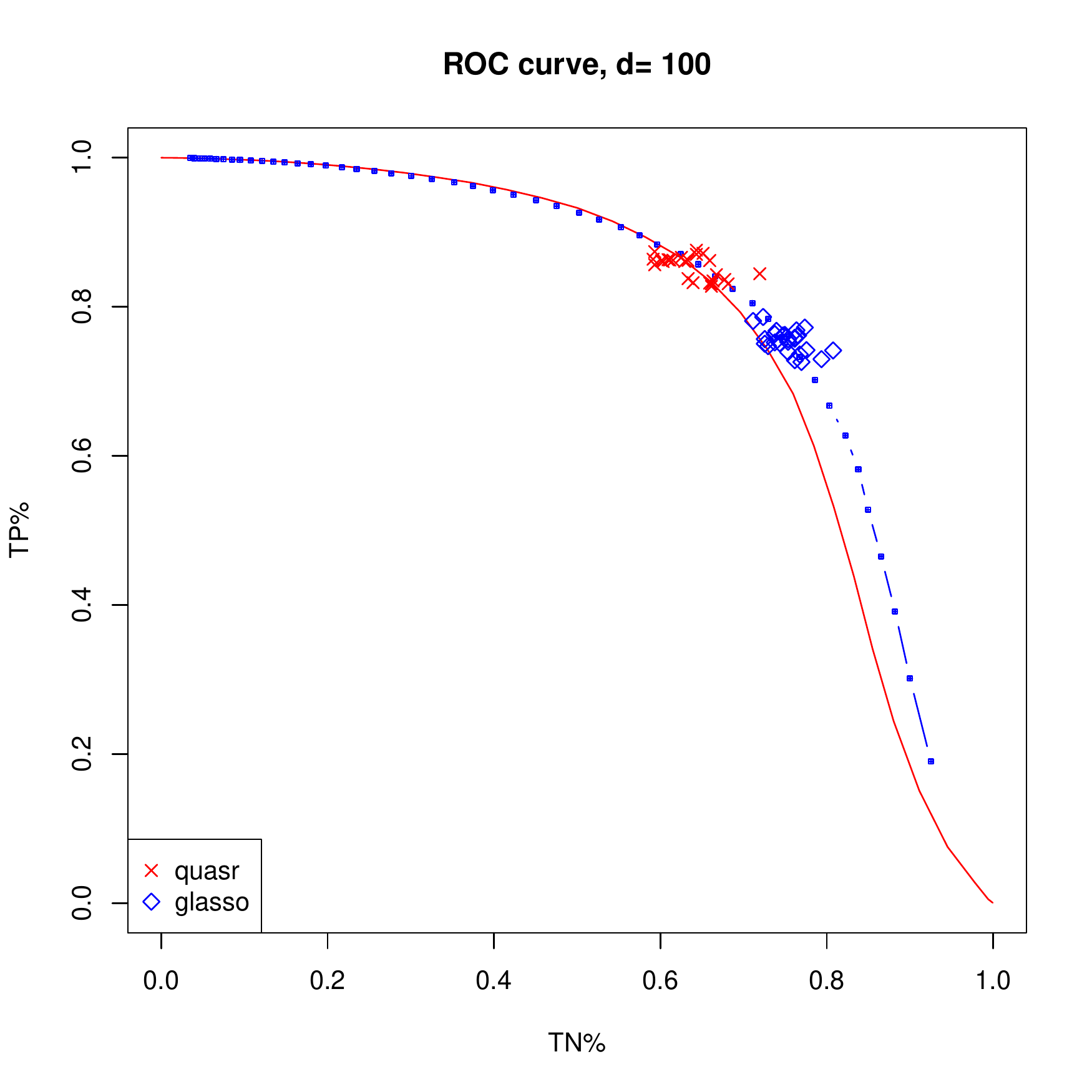}
\caption{ROC curve, Erdos-Renyi graph. Solid line: Gaussian QUASR; dotted line: glasso.} \label{roc2}
\end{figure}
Tables \ref{edgesel1} and \ref{edgesel2} displays true positive and true negative rates for varying choices of $d$, fixing $n=100$. The parentheses are the standard deviation for 25 repetitions of the experiment. As are suggested by the ROC curves, score matching prefers a higher sensitivity and lower specificity to the regularized MLE, and for simulations when $d$ is relatively small compared to $n$, score matching has a significantly higher true positive rate with only negligible reduction in true negative rate.

\begin{table}[ht]
\centering
\begin{tabular}{lll}
  \hline
 & quasr & glasso  \\
  \hline
d=30 \%TN & 0.941 (0.043) & 0.975 (0.025)  \\
  d=30 \%TP & 0.778 (0.051) & 0.674 (0.042)  \\
  d=75 \%TN & 0.821 (0.037) & 0.876 (0.042)  \\
  d=75 \%TP & 0.892 (0.017) & 0.832 (0.017)  \\
  d=120 \%TN & 0.751 (0.052) & 0.816 (0.044)  \\
  d=120 \%TP & 0.922 (0.01) & 0.874 (0.009)  \\
  d=150 \%TN & 0.699 (0.049) & 0.776 (0.052)  \\
  d=150 \%TP & 0.936 (0.006) & 0.899 (0.012)  \\
   \hline
\end{tabular}
\caption{Edge selection accuracy, tree graph.} \label{edgesel1}
\end{table}

\begin{table}[ht]
\centering
\begin{tabular}{lll}
  \hline
 & quasr & glasso  \\
  \hline
d=30 \%TN & 0.969 (0.028) & 0.986 (0.02)  \\
  d=30 \%TP & 0.743 (0.036) & 0.612 (0.039) \\
  d=50 \%TN & 0.878 (0.028) & 0.927 (0.025)  \\
  d=50 \%TP & 0.785 (0.027) & 0.673 (0.032)  \\
  d=100 \%TN & 0.633 (0.029) & 0.754 (0.02)  \\
  d=100 \%TP & 0.858 (0.013) & 0.753 (0.013)  \\
  d=150 \%TN & 0.486 (0.019) & 0.634 (0.018)  \\
  d=150 \%TP & 0.892 (0.01) & 0.791 (0.011)  \\
   \hline
\end{tabular}
\caption{Edge selection accuracy, ER graph} \label{edgesel2}
\end{table}

\subsubsection{Computation}

In Figure \ref{cputree} we compare the runtime of our algorithm with glasso. We simulate random Gaussian tree data with $n=100$, $d=50$. We fit over a path of $\lambda$s and plot runtime against number of selected edges. More regularization results in sparser graphs, and so convergence is faster. In this experiment our method is much faster than glasso, sometimes by a factor of 4 or more. The gap narrows for sparse estimated graphs. This is because while our algorithm is written efficiently in $\tt{C++}$, it doesn't (yet) utilize sparse matrix libraries, while glasso does. Since our coordinate-wise descent algorithm involves sparse inner products, we believe our runtimes can be improved in the sparse regime.

\begin{figure}
\centering
\includegraphics[scale=.9]{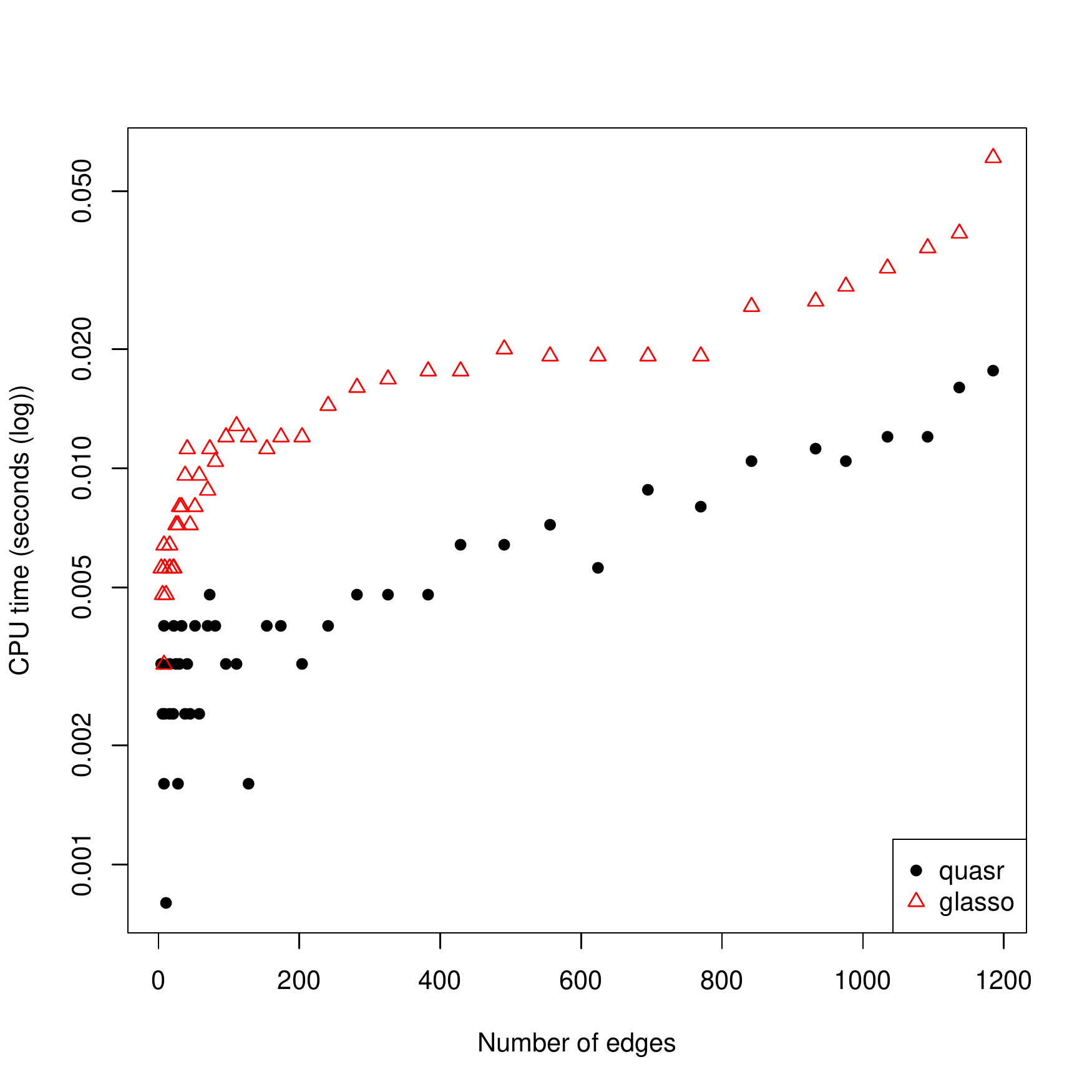}
\caption{Runtime, Gaussian QUASR and glasso. Gaussian data, n=100, d=50}. \label{cputree}
\end{figure}
\subsection{Nonparametric Score Matching}

\subsubsection{Risk and Density Estimation}
In this section we compare score matching and MLE when the sufficient statistics are chosen to be Legendre polynomials. 
%We fix $m_1=m_2=4$. 
For each choice of $n$, we simulate non-Gaussian data whose density factorizes as a tree. We do this by marginally applying the transformation $y=\text{sign}(x-0.5) \mid x - 0.5 \mid ^{0.6}/5 + 0.5$ to Gaussian data which has a tree factorization, which has been scaled to have means 0.5 and covariance $1/8^2$, so it fits in the unit cube; the resulting data follows a Gaussian copula distribution, which is also a pairwise distribution. We train the model using both regularized MLE and score matching under constraint that it factorizes according to the given tree, and we  estimate along a path of $\lambda$s and choose the regularization parameter $\lambda$ to minimize the held-out risk. Since the density has a (known) tree factorization, it is possible to compute the likelihood (and hence the MLE) exactly using functional message passing (\cite{janofsky2015exponential}, chapter 3). It is also possible to compute the marginals using the same message passing algorithm.

Figure \ref{nprisk} displays the held-out risk for both methods, with sample size varying from 24 to 5000 and $d=20$. We average over 5 replications of the experiment. The MLE outperforms the score matching estimator, but the score matching estimators performance greatly improves relative to the MLE as $n$ increases. 

Figure \ref{npcontour} displays contours from one bivariate marginal from the aforementioned simulation. The top row shows the MLE and score matching estimator for $n=24$, and the bottom for $n=182$. For $n=24$, the score matching marginal can make out  much of the distinguishing features of the density such as the multiple modes, but isn't as informative as the MLE. At $n=182$ the marginals appear almost identical.
\begin{figure}
\centering
\includegraphics[scale=.9]{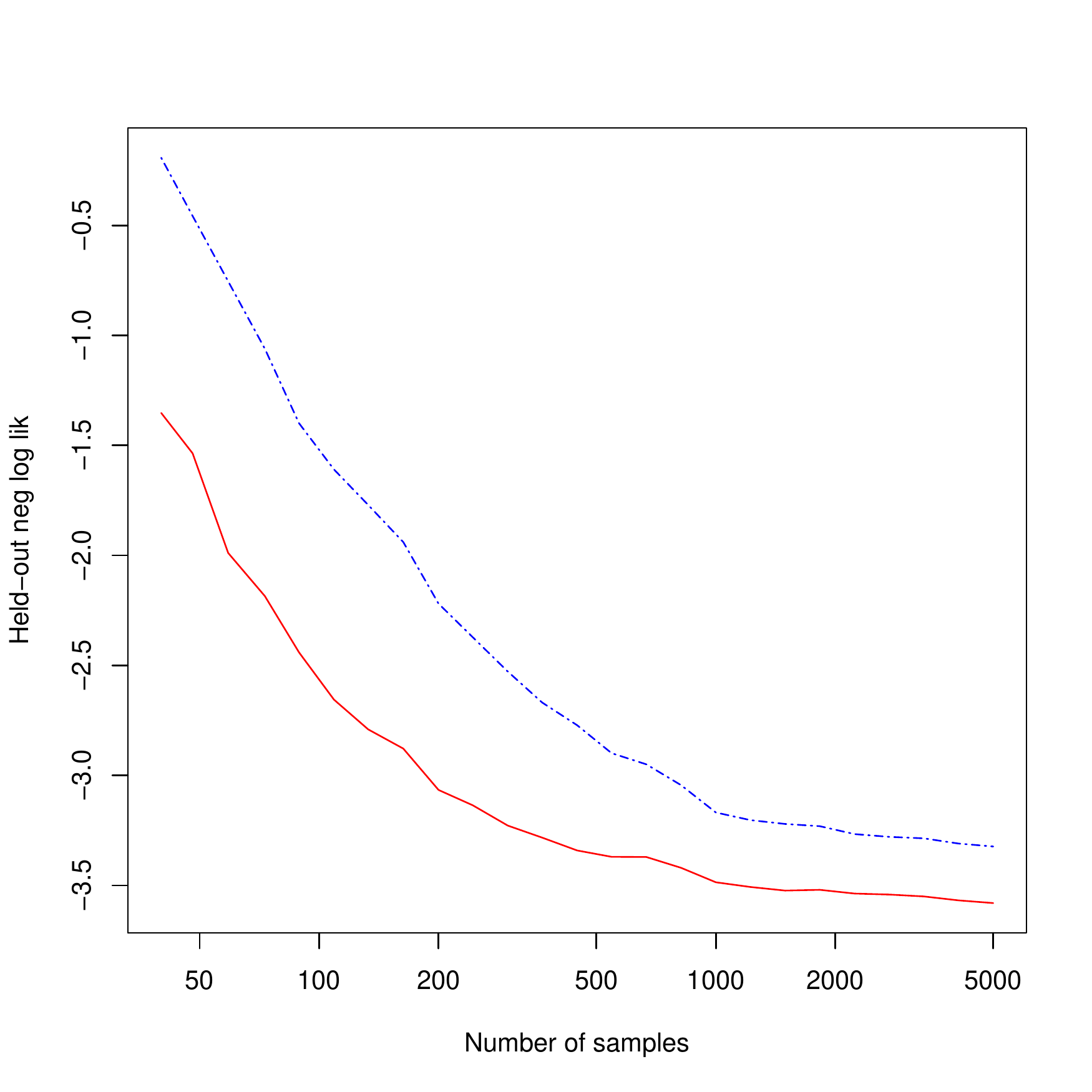}
\caption{Held-out negative log likelihood for different training sizes. Red: regularized MLE; Blue dashed: score matching. Data generated from a non-Gaussian tree distribution.} \label{nprisk}
\end{figure}
We conclude that while MLE is more efficient as may be expected, score matching performs quite well, especially with larger sample sizes. 
\begin{figure}
\centering
\includegraphics[page=1,scale=.4]{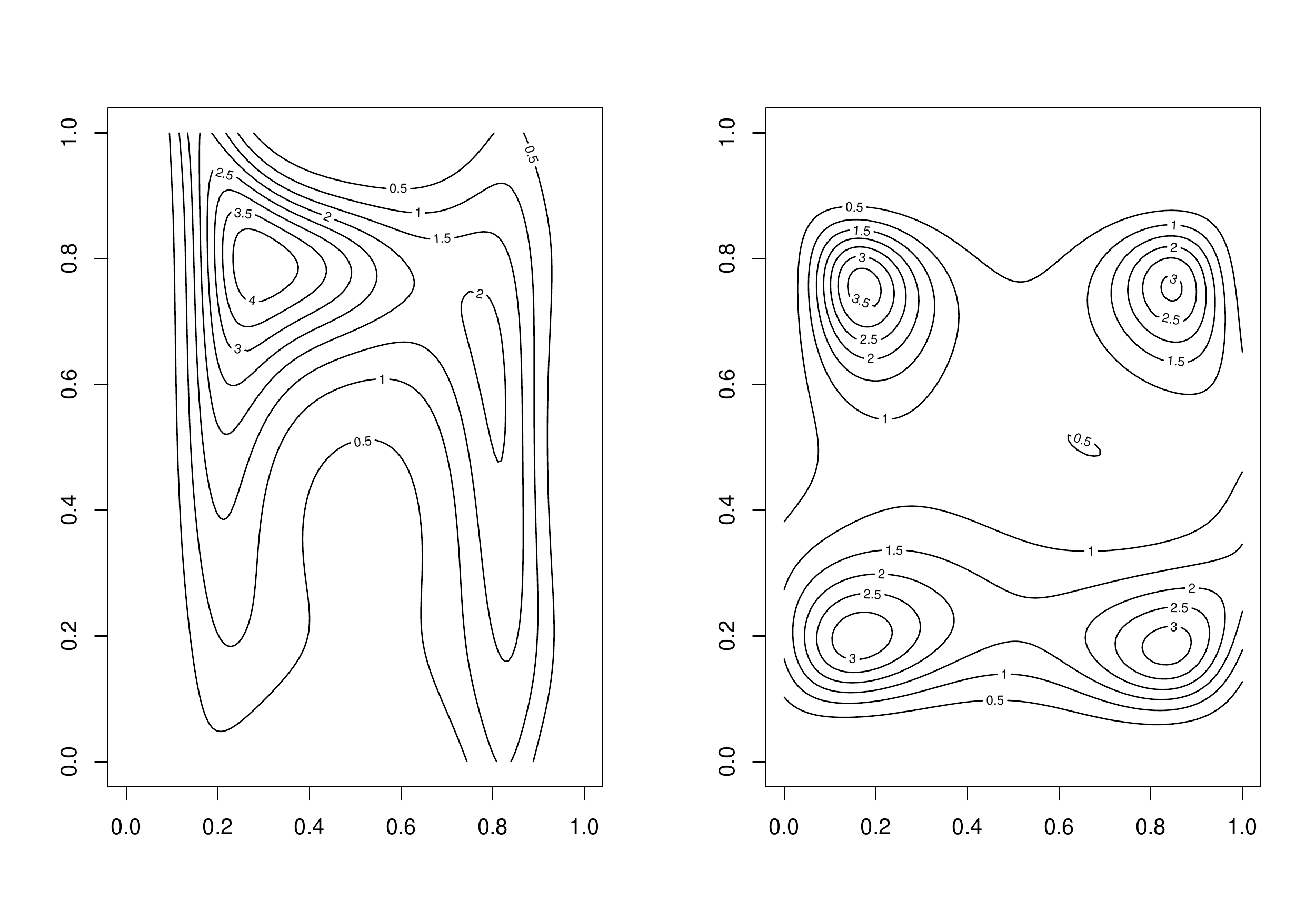}
\includegraphics[page=2,scale=.4]{bivdenrisk.pdf}
\includegraphics[page=3,scale=.4]{bivdenrisk.pdf}

\caption{Estimated bivariate marginal from non-Gaussian tree data, d=20. Left: regularized MLE; Right: Score matching. Top: n=40. Middle: n=244 Bottom: n=1828.} \label{npcontour}
\end{figure}
Furthermore, we emphasize that at this cost of statistical efficiency, score matching can be computed easily under any graph structure, even when the likelihood is not tractable, while MLE is typically not tractable and must be approximated.

\subsubsection{ROC Curves}

Figures \ref{rocquasr1} and \ref{rocquasr3} display ROC curves from four experiments. We simulate data with $d=20$ and $n$ either 30 or 100. In the first two experiments, the data is Gaussian; in the last two, it is non-Gaussian (copula). The graph is either a random spanning tree with $d-1$ edges, or a graph with each possible edge having inclusion probability $0.2$, or expected number of edges $d(d-1)*.1$. The ROC curves trace the true positive and true negative rates, varying the value of $\lambda$. The curves are averaged over 10 repetitions of the experiment (with the data i.i.d. from the same distribution). We choose $m_1,m_2$ to maximize $\max_\lambda TP(\lambda)+TN(\lambda)$. We compare our method to the SKEPTIC estimator from \citep{liu2012high},  which was designed in particular for model selection for copula graphical models, and a tree-reweighted approximation to the MLE from \cite{janofsky2015exponential}. Our experiments show our method performs just as well or slightly worse than the competing methods, despite the fact that SKEPTIC is specifically tailored for Gaussian and Gaussian copula graph learning.

\begin{figure}
\centering
\includegraphics[page=1,scale=.5]{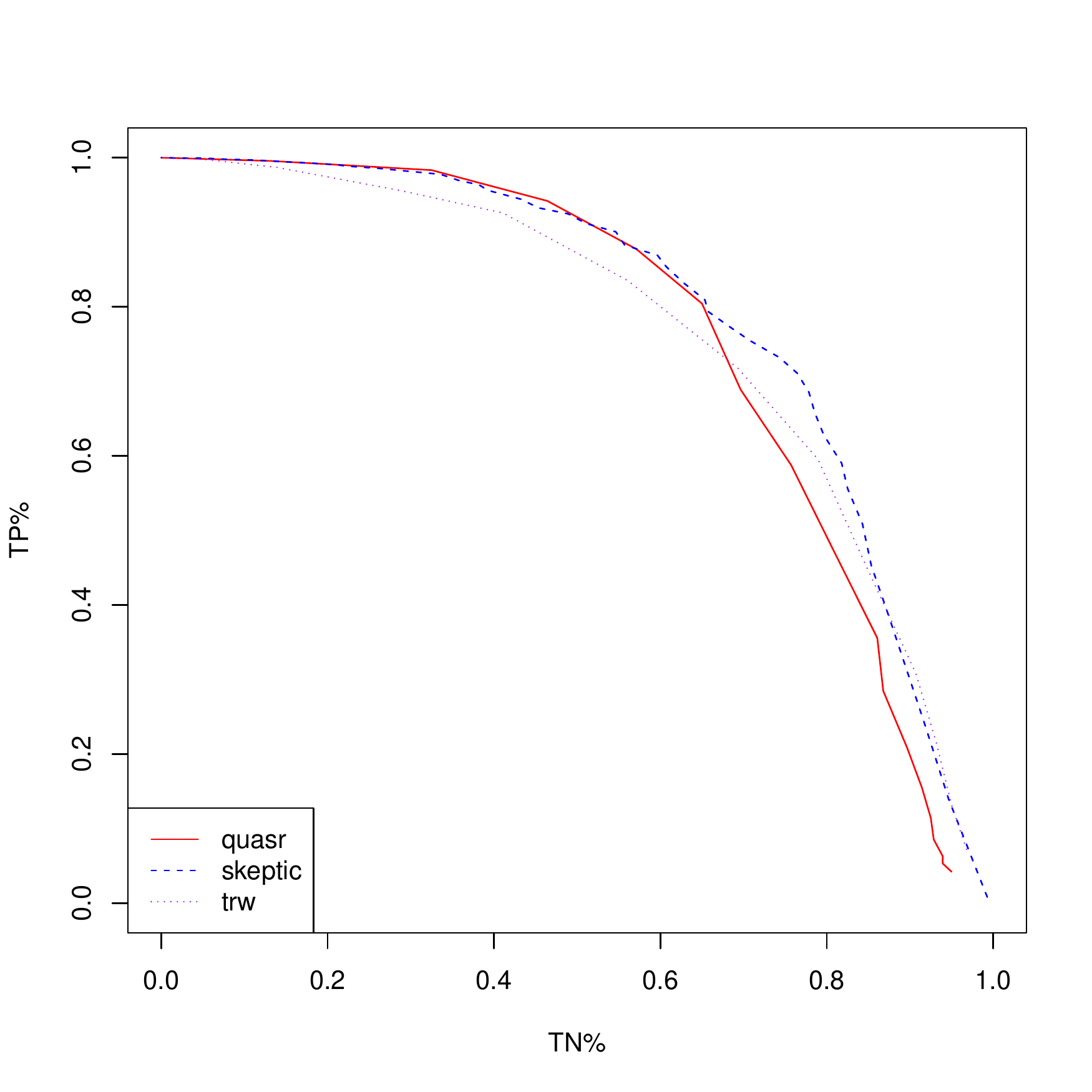}
\includegraphics[page=1,scale=.5]{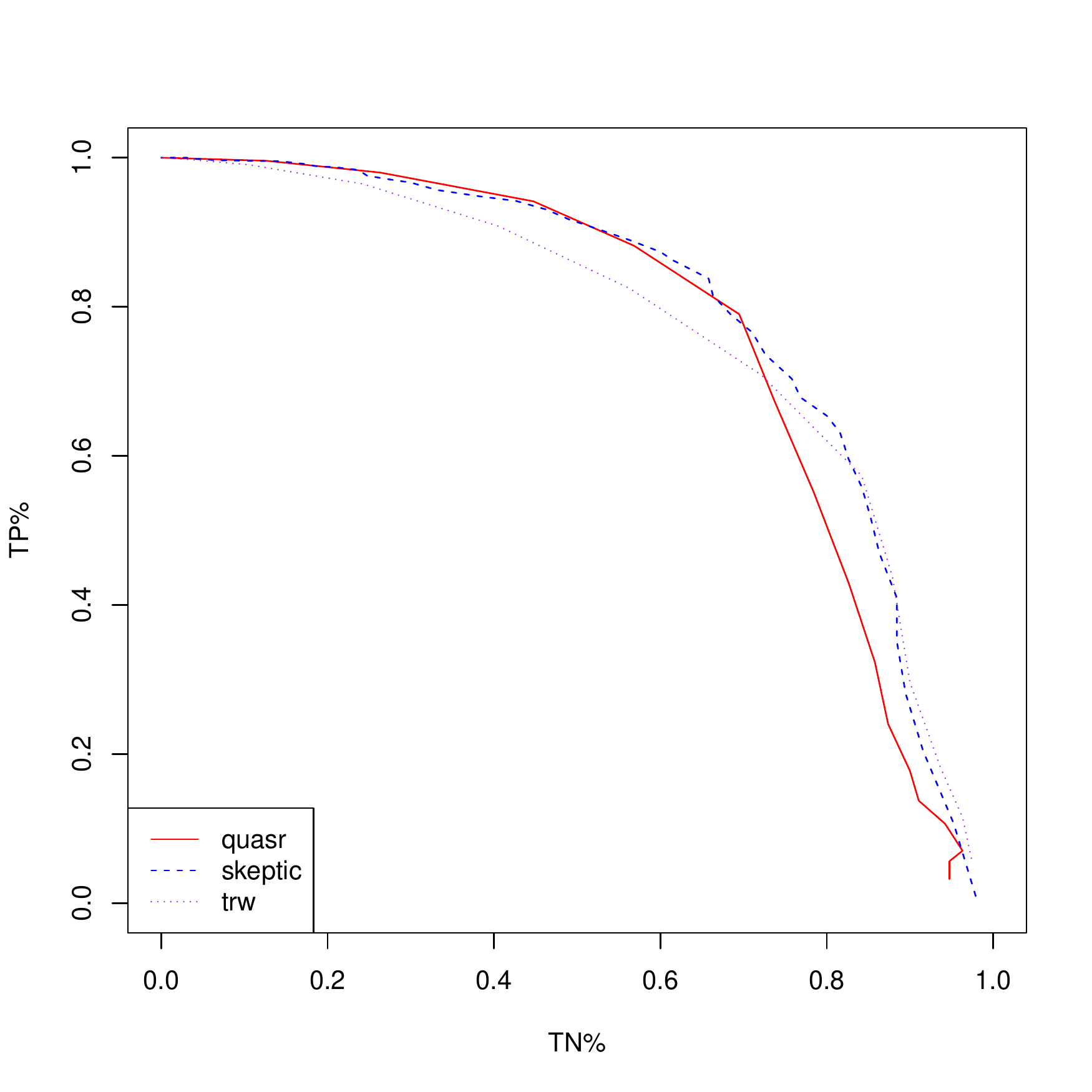}

\caption{ROC Curve, Gaussian data, n=30, d=20. Top: ER graph. Bottom: Tree graph.}\label{rocquasr1}
\end{figure}

%\begin{figure}
%\centering
%\includegraphics[page=1,scale=.5]{/Users/eric/Dropbox/quasr/rocquasrtreegauss.pdf}
%\caption{ROC Curve, tree graph, Gaussian data. n=30, d=20.}\label{rocquasr2}
%\end{figure} 

\begin{figure}
\centering
\includegraphics[page=1,scale=.5]{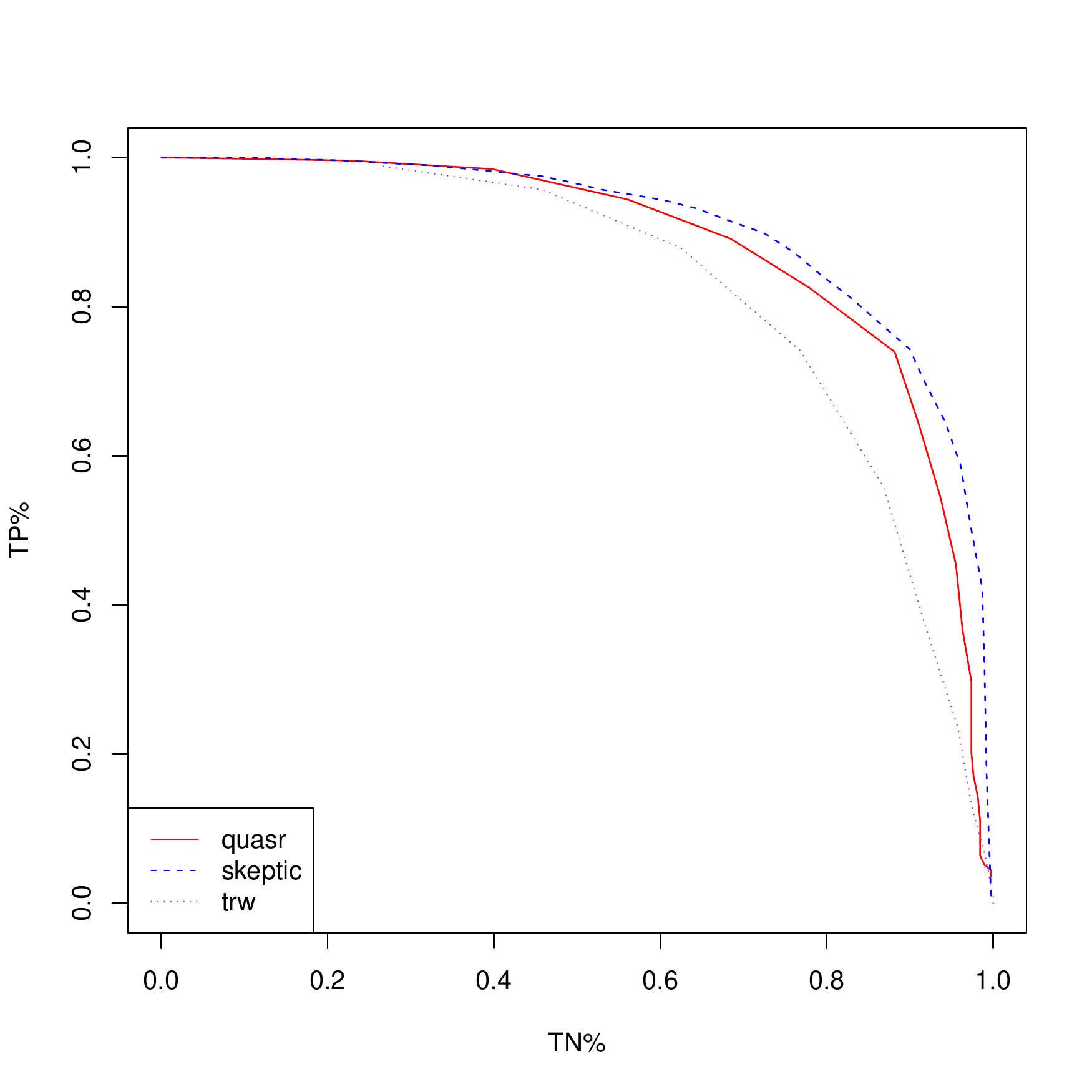}
\includegraphics[page=1,scale=.5]{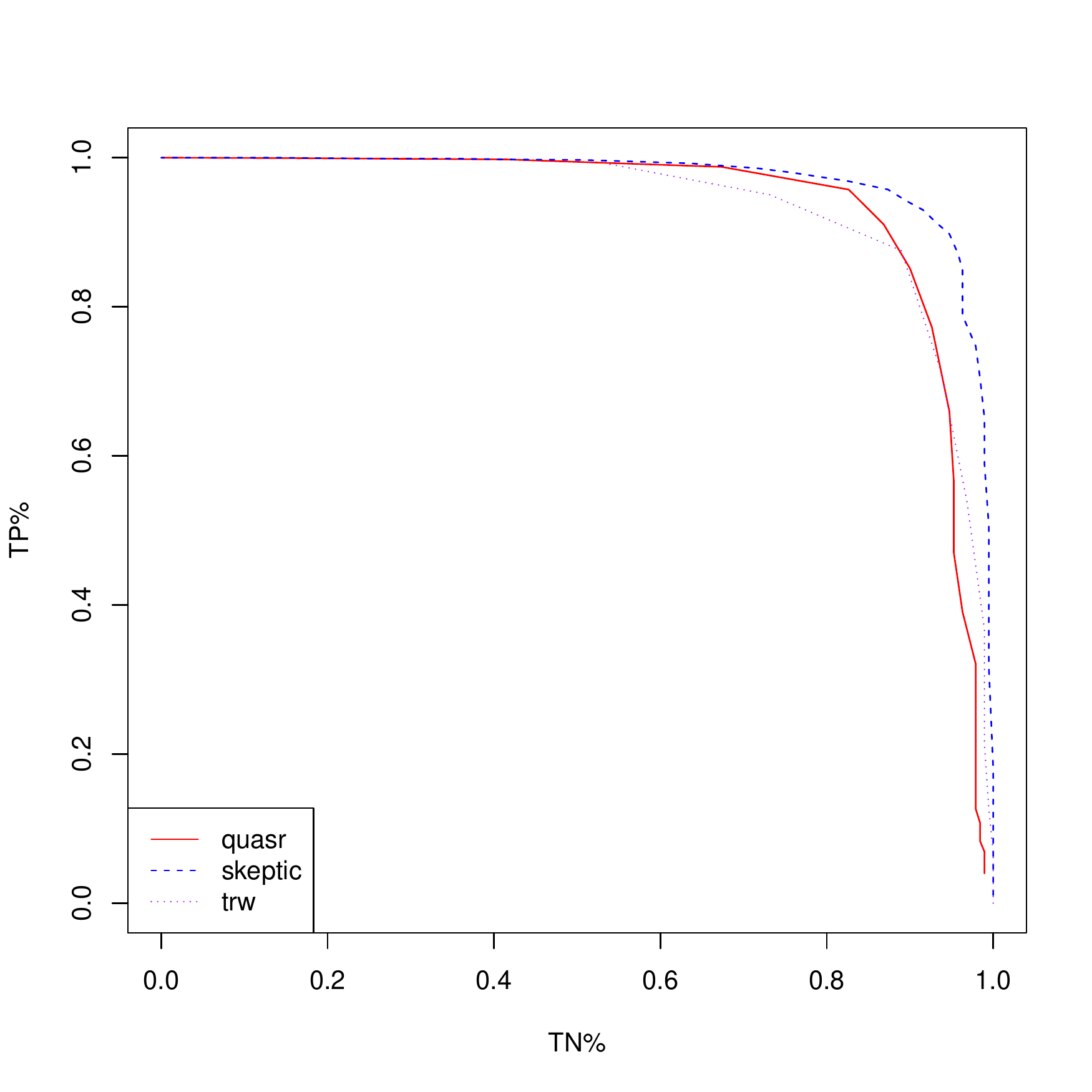}
\caption{ROC Curve, non-Gaussian data, n=100, d=20. Top: ER graph. Bottom: Tree graph.}\label{rocquasr3}
\end{figure}

%\begin{figure}
%\centering
%\includegraphics[page=1,scale=.5]{/Users/eric/Dropbox/quasr/rocquasrtreenongauss.pdf}
%\caption{ROC Curve, tree graph, non-Gaussian data. n=100, d=20}\label{rocquasr4}
%\end{figure}

\section{Discussion}

This work introduces a new approach to estimating pairwise, continuous, exponential family graphical models. Since the normalizing constant for these models is usually intractable, we propose a new scoring rule which obviates the need for computing it. Our resulting estimator may be expressed as a second-order cone program. We show consistency and edge selection results for this estimator, including as special cases a new method for precision matrix estimation, and for nonparametric edge selection with exponential series. We propose algorithms for solving the convex problem which are highly scalable and amenable to parallelization.

This work raises several areas for further work. The regularized MLE, which is not tractable, has been shown to have stronger guarantees for edge selection for the nonparametric graphical model,  and it's an open question whether a computationally tractable method can achieve these bounds. There exist other proper scoring rules \cite{dawid2014theory}, few of which have been explored or understood in the context of learning exponential family graphical models.

\section{Proofs}

\begin{proof}[Proof of Proposition \ref{boundedscorematch}]

Consider the functional

\begin{align}
	J(q) &:= \mathbb{E}_p\left[(\Vert\nabla \log p-\nabla \log q )\otimes x(1-x)\Vert_2^2\right].
\end{align}

If $J(q)=0$, then it must be $\nabla \log p= \nabla \log q$  a.e., because their integrated squared distance is zero with respect to a weight function which is nonzero a.e. This implies $\log q = \log p + c$ a.e. for some constant $c$, but $c=0$ because $p$ and $q$ must both integrate to one. Furthermore, $J$ is non-negative so it is minimized when $q=p$. If $p$ and $q$ belong to an exponential family with respective natural parameters $\theta$ and $\theta'$, $\theta=\theta'$ when the family is minimal.

Now,

\begin{align}
	J(q) &= \mathbb{E}_p\left[\Vert\nabla \log q \otimes x(1-x)\Vert_2^2\right] \\
	&\qquad+ 2\mathbb{E}_p\left[\sum_{i\in V} (\nabla_i \log q \cdot \nabla_i \log p)\otimes x(1-x)\Vert_2^2\right]+  constant, \notag
\end{align}

the constant not depending on $q$. We have, by integration by parts,

\begin{align}
	\mathbb{E}_p\left[(\nabla_i \log q \nabla_i \log p)x_i(1-x_i)\right] &=\intop p_i (\nabla_i \log q \nabla_i \log p)x_i(1-x_i) \\
	&= \intop p_i \frac{\nabla_i p_i}{p_i}(\nabla_i \log q )x_i(1-x_i) \\
	&= p_i(x_i)(\nabla_i \log q x_i(1-x_i))\bigg]_{x_i=1} \\
	&\qquad-p_i(x_i)(\nabla_i \log q x_i(1-x_i))\bigg]_{x_i=0} \notag\\
	&\qquad- \intop p_i \nabla_i (\nabla_i \log q x_i(1-x_i)) \notag\\
	&= - \intop p_i \nabla_i (\nabla_i \log q x_i(1-x_i)), 
\end{align}

where in the last line we applied the boundary assumption. Thus, we see that $J(q)$ is equal to $\mathbb{E}_p\left[h(X,q)\right]$ plus some terms which don't depend on $q$, so from the argument above $\mathbb{E}_p\left[ h(X,q)\right]$ is minimized when $p=q$. We conclude that $h$ is a proper scoring rule.

%Suppose we apply the marginal transformation $Y_i = f(X_i):= \text{logit}(X_i)=\log\left(\frac{X_i}{1-X_i}\right)$.  $f^{-1}(y)=\frac{e^y}{1+e^y}$, the logistic function, and $\frac{\partial f^{-1}}{\partial x}=f^{-1}(x)(1-f^{-1}(x)$, and $\frac{\partial f}{\partial y}=\frac{1}{y(1-y)}$. Then the joint distribution of $Y$ is given by
%
%\begin{align}
%	p_y(y) &= \prod_{i\in V} \left(f^{-1}(y_i)(1-f^{-1}(y_i))\right)^{-1}p_x(f^{-1}(y))
%\end{align}
%
%For $i\in V$ we have
%
%\begin{align}
%	\frac{\partial \log p_y(y)}{\partial y_i} &= -\left(\frac{e^{y_i}-1}{e^{y_i}+1}\right)+\left(\frac{\partial \log p_x(f^{-1}(y))}{\partial y_i}\right)f^{-1}(y_i)(1-f^{-1}(y_i)\\
%	&= 2x_i-1 + \frac{\partial \log p_x(x)}{\partial x_i} x_i(1-x_i),
%\end{align}
%giving
%
%\begin{align}
%	\left(\frac{\partial \log p_y(y)}{\partial x_i}\right)^2 &= \left(2x_i-1\right)^2 + \left(\frac{\partial \log p_x(x)}{\partial x_i}\right)^2 x_i^2(1-x_i)^2 \\
%	& + 2(2x_i-1)x_i(1-x_i)\frac{\partial \log p_x(x)}{\partial x_i}
%\end{align}
%and
%
%\begin{align}
%	\frac{\partial^2 \log p_y(y)}{\partial y_i^2} &= \frac{2e^{y_i}}{(1+e^{y_i})^2} + \frac{\partial^2 \log p_x}{\partial x_i^2} f^{-1}(y_i)(1-f^{-1}(y_i))+ \frac{\partial \log p_x}{\partial x_i}\frac{e^{y_i}(e^{y_i}-1)}{(1+e^{y_i})^3} \\
%	&= x_i(1-x_i)\frac{\partial^2 \log p_x}{\partial x_i^2} + 2x_i(1-x_i)+ (2x_i-1)x_i(1-x_i)\frac{\partial \log p_x}{\partial x_i}
%\end{align}

\end{proof}

\section{Parameter Estimation}

\begin{lem}\label{parconsistlem}
	If $n\geq Cmd$ and $\lambda_n\geq 2\mathcal{R}^*((\widehat{\Gamma}-\Gamma)\theta^*+\widehat{K}-K)$, with probability at least $1-2d\exp\left\{-\frac{\bar{\epsilon}^2}{4\underline{\epsilon}^2}md\right\}$, the regularized score matching estimator $\widehat{\theta}$ satisfies
	
\begin{align}
	\Vert\widehat{\theta}-\theta^*\Vert_2 &\leq \frac{7\lambda_n}{\underline{\epsilon}}\sqrt{d+\left| E\right|}.
\end{align}
\end{lem}

\begin{proof}
Define the function

\begin{align}
	\mathcal{E}(\delta) &= \mathcal{L}(\theta^*+\delta)-\mathcal{L}(\theta^*)+\lambda_n(\mathcal{R}(\theta^*+\delta)-\mathcal{R}(\theta^*))  \\
	&=\frac{1}{2}(\theta^*+\delta)^\top\widehat{\Gamma}(\theta^*+\delta)+(\theta^*+\delta)^\top\widehat{K}-\frac{1}{2}(\theta^*)^\top\widehat{\Gamma}(\theta^*)^\top - (\theta^*)^\top\widehat{K} \\
	&\qquad+ \lambda_n(\mathcal{R}(\theta^*+\delta)-\mathcal{R}(\theta^*))\notag \\
	&=\frac{1}{2}\delta^\top\widehat{\Gamma}\delta +  \delta^\top(\widehat{\Gamma}\theta^*+\widehat{K})  + \lambda_n(\mathcal{R}(\theta^*+\delta)-\mathcal{R}(\theta^*)) \\
	&= \frac{1}{2}\delta^\top\widehat{\Gamma}\delta +  \delta^\top(\widehat{\Gamma}\theta^*-\Gamma\theta^*+\widehat{K}-K)  + \lambda_n(\mathcal{R}(\theta^*+\delta)-\mathcal{R}(\theta^*)).
\end{align}

Since $\mathcal{E}(0)=0$, it must be that $\mathcal{E}(\widehat{\delta})\leq 0$.

Using the sub-Gaussian assumption, we may apply \citep{vershynin2010introduction} Remark 5.51, which says for any $c\in(0,1),t\geq 1$, with probability at least $1-2\exp\{-t^2md\}$, if $n\geq C(t/c)^2 md$, then for any $i\in V$, and any vector $\delta_i$
\begin{align}
	\delta_i^\top \widehat{\Gamma}_i\delta_i \geq \delta_i^\top \Gamma_i\delta_i + \bar{\epsilon} c\Vert\delta_i\Vert^2,
\end{align}

setting $c=\frac{\underline{\epsilon}}{2\bar{\epsilon}}$, and $t=1/c$ we get that if $n\geq Cd$, with probability at least $1-2\exp\left\{\frac{\bar{\epsilon}^2}{4\underline{\epsilon}^2}d\right\}$,

\begin{align}
	\frac{\underline{\epsilon}}{2}\Vert\delta_i\Vert^2 \leq \delta_i^\top\widehat{\Gamma}_i\delta_i.
\end{align}

Applying the union bound over all $i\in V$, with probability at least $1-2d\exp\left\{\frac{\bar{\epsilon}^2}{4\underline{\epsilon}^2}d\right\}$,

\begin{align}
	\frac{\underline{\epsilon}}{2}\delta^\top\delta \leq \delta^\top\widehat{\Gamma}\delta,
\end{align}

where $\delta=(\delta_1^\top,\ldots,\delta_d^\top)^\top$.
%Now, since $\widehat{\Gamma}=\widehat{\bar{\Gamma}}\otimes I_d$, the eigenvalues of $\widehat{\Gamma}$ are identical to those of $\widehat{\bar{\Gamma}}$, replicated $d$ times. It follows that if $n\geq Cd$, with probability at least $1-2\exp\{\frac{\bar{\epsilon}^2}{4\underline{\epsilon}^2}d\}$,
%
%\begin{align}
%	\frac{\underline{\epsilon}}{2}\delta^\top\delta \leq \delta^\top\widehat{\Gamma}\delta.
%\end{align}

By (generalized) Cauchy-Schwarz,

\begin{align}
	\left|\langle\delta,(\widehat{\Gamma}-\Gamma)\theta^*+\widehat{K}-K\rangle\right| &\leq\mathcal{R}(\delta)\mathcal{R}^*((\widehat{\Gamma}-\Gamma)\theta^*+\widehat{K}-K) \\
	&\leq \frac{\lambda_n}{2}(\mathcal{R}(\delta_\mathcal{P}+\mathcal{R}(\delta_{\mathcal{P}^\perp}).\label{gencs}
\end{align}

where $\delta_A$ denotes the projection of $\delta$ onto the set $A$. From  \citep{negahban2012unified} Lemma 3, because $\mathcal{R}$ is decomposable, it  holds that

\begin{align}
	\mathcal{R}(\theta^*+\delta)-\mathcal{R}(\theta^*) \geq \mathcal{R}(\delta_{\tilde{\mathcal{P}}^\perp})-\mathcal{R}(\delta_{\tilde{\mathcal{P}}}). \label{negh}
\end{align}

Combining \eqref{gencs} and \eqref{negh},

\begin{align}
&\langle \delta,(\widehat{\Gamma}-\Gamma)\theta^*+\widehat{K}-K\rangle + \frac{\lambda_n}{2}\left(\mathcal{R}(\theta^*+\delta)-\mathcal{R}(\theta^*) \right) \\&\geq -\frac{3\lambda_n}{2}\mathcal{R}(\delta_{\tilde{\mathcal{P}}})-\frac{\lambda_n}{2}\mathcal{R}(\delta_{\tilde{\mathcal{P}}^\perp}) \geq -\frac{3\lambda_n}{2}\mathcal{R}(\delta_{\tilde{P}})\label{eq:proof1}.
\end{align} 
Using the subspace compatibility constant we have that 
\begin{align}\mathcal{R}(\delta_{\tilde{\mathcal{P}}})\leq \sqrt{d+\left| E\right|} \Vert\delta_{\tilde{\mathcal{P}}}\Vert \leq \sqrt{d+\left| E \right|}\Vert\delta\Vert.
\end{align}
Thus conditioning on the aformentioned probability,

\begin{align}
	\mathcal{E}(\delta) &\geq \frac{\underline{\epsilon}}{4} \Vert\delta\Vert^2- \frac{3\lambda}{2}\Vert\delta\Vert\sqrt{d+E} \\
	&=\Vert\delta\Vert\left(\frac{\underline{\epsilon}}{4}\Vert\delta\Vert-\frac{3\lambda_n}{2}\sqrt{d+E}\right). \label{ineq}
\end{align}

Now, consider the set 

\begin{align}
\mathcal{C} = \left\{\delta: \Vert\delta\Vert\leq\frac{7\lambda_n}{\underline{\epsilon}}\sqrt{d+E}\right\}.
\end{align}

$\mathcal{C}$ is a compact, convex set. Furthermore, for all $\delta\in\partial\mathcal{C}$, from \eqref{ineq} we see that $\mathcal{E}(\delta)>0$. Also observe that $0\in\text{int}\mathcal{C}$. Since $\mathcal{E}(\widehat{\delta})\leq 0$, it must follow that $\widehat{\delta}\in\text{int}\mathcal{C}$, in other words

\begin{align}
	\Vert\widehat{\theta}-\theta^*\Vert_2 &\leq \frac{7\lambda_n}{\underline{\epsilon}} \sqrt{d+E}.
\end{align}

\end{proof}

\begin{proof}[Proof of Theorem \ref{parconsist}]
Applying a concentration bound to $\widehat{K}_{ij}^u-K_{ij}^u$ in addition to a union bound, we have that $\Vert\widehat{K}-K\Vert_{\max}>t$ with probability no more than $\exp\{2\log(md)-c_2nt^2\}$ for $t\leq\nu$, for constants $c_2,\nu>0$. Similarly,

\begin{align}
	\Vert(\widehat{\Gamma}-\Gamma)\theta^*\Vert_{\max}&\leq 2\kappa_{\theta^*,1}\Vert\widehat{\Gamma}-\Gamma\Vert_{\max},
\end{align}

and $\Vert\widehat{\Gamma}-\Gamma\Vert_{\max}>t$ with probability no more than $\exp\{2\log(md)-c_1nt^2\}$ for $t<\nu_2$ for $c_1,\nu_2>0$. Furthermore, observe that for any vector $\theta$,

\begin{align}
	\mathcal{R}^*(\theta) &\leq \sqrt{m} \Vert\theta\Vert_{\max},
\end{align}

thus setting $\lambda_n\asymp \sqrt{\frac{mk_{1,\theta}^2\log(md)}{n}}$, the conditions in Lemma \ref{parconsistlem} will be satisfied with probability approaching one.
\end{proof}

\section{Model Selection}

Our proof technique is the \emph{primal dual witness} method, used previously in analysis of model selection for graphical models \citep{modelselecttaylor2014,ravikumar2011high}. It proceeds as follows: construct a primal-dual pair $(\widehat{\theta},\widehat{Z})$ which satisfies $\text{supp}(\widehat{\theta})=\text{supp}(\theta^*)$, and also satisfies the stationary conditions for \ref{scorematch} with high probability. The stationary conditions for \ref{scorematch} are
\begin{align}
	\widehat{\Gamma}\widehat{\theta}+\widehat{K} + \lambda_n\widehat{Z} = 0,
\end{align}

where $\widehat{Z}$ is an element of the subdifferential $\partial\mathcal{R}(\widehat{\theta}_e)$:

\begin{align}
	(\partial\mathcal{R}(\theta_e))_{ij} = \begin{cases}
		\{\theta_{ij}: \Vert\theta_{ij}\Vert_2\leq 1\}, & \text{  if } \theta_{ij}=0; \\
		\frac{\theta_{ij}}{\Vert\theta_{ij}}\Vert_2, & \text{  if } \theta_{ij}\not=0.
	\end{cases}
\end{align}
From this we may conclude that there exists a solution to \ref{scorematch} that is sparsistent. In particular, we have the following steps:
\begin{enumerate}
	\item Set $\widehat{\theta}_{E^c} = 0$;
	\item Set $\widehat{Z}_{ij} = \partial\mathcal{R}(\theta^*_e)_{ij}=\frac{\bar{\theta}^*_{ij}}{\Vert\bar{\theta}^*_{ij}\Vert}$ for $(i,j)\in E$;
	\item Given these choices for $\widehat{\theta}_{E^c}$ and $\widehat{Z}_E$, choose $\widehat{\theta}_E$ and $\widehat{Z}_{E^c}$ to satisfy the stationary condition \ref{scorematch}.
\end{enumerate}

For our procedure to succeed, we must show this primal-dual pair $(\widehat{\theta},\widehat{Z})$ is optimal for \ref{scorematch}, in other words

\begin{align}
	&\widehat{\theta}_{ij} \not = 0 , &\text{for}\qquad (i,j)\in E;\label{pdscm1} \\
	&\Vert\widehat{Z}_{ij}\Vert <1, &\text{for}\qquad (i,j)\not\in E. \label{pdscm2}
\end{align}

In the sequel we show these two conditions hold with probability approaching one.

\begin{lem}\label{primaldual1}
	Suppose that $\Vert\widehat{\Gamma}_{EE}-\Gamma_{EE}\Vert_{\max}\leq\frac{1}{2ms\kappa_\Gamma}$. Then there exists a solution to \eqref{scorematch}, $\widehat{\theta}$, satisfying
	
\begin{align}
	\Vert\widehat{\theta}_{E}-\theta^*_{E}\Vert_\infty &\leq 2\kappa_\Gamma\left(2\kappa_{1,\theta}\Vert(\widehat{\Gamma}_{EE}-\Gamma_{EE})\Vert_{\max}+\Vert\widehat{K}-K\Vert_\infty+\lambda_n/\sqrt{m}\right).
\end{align}
\end{lem}

\begin{proof}
The stationary condition for $\widehat{\theta}_E$, observing that $\widehat{\theta}_{E^c}=\theta^*_{E^c}=0$, is given by

\begin{align}
	\widehat{\Gamma}_{EE}\widehat{\theta}_E+\widehat{K}_E + \lambda_n \widehat{Z}_E&=0.
\end{align}

Re-arranging and observing that $\Gamma_{EE}\theta^*_E=-K_E$, we have

\begin{align}
	\widehat{\Gamma}_{EE}\widehat{\theta}_E+\widehat{K}_E + \lambda_n \widehat{Z} &=\widehat{\Gamma}_{EE}\widehat{\theta}_E-\Gamma_{EE}\theta^*_E+\widehat{K}_E -K_E+ \lambda_n \widehat{Z}_E \\
	&=(\widehat{\Gamma}_{EE}-\Gamma_{EE})\widehat{\theta}+\Gamma_{EE}(\widehat{\theta}_E-\theta^*_E)+\widehat{K}_E-K_E+\lambda_n\widehat{Z}_E.
\end{align}

Consider the map

\begin{align}
	F(\Delta_E) = -\Gamma_{EE}^{-1}\left((\widehat{\Gamma}_{EE}-\Gamma_{EE})(\Delta_E+\theta^*_E)+\Gamma_{EE}\Delta_E+\widehat{K}_E-K_E+\lambda_n\widehat{Z}_E\right)+\Delta_E.
\end{align}

$F$ has a fixed point $F(\Delta_E)=\Delta_E$ at $\widehat{\Delta}_E=\widehat{\theta}_E-\theta^*_E$ for any solution $\widehat{\theta}$. Define $\tilde{r}:=2\kappa_\Gamma\left(2\kappa_{1,\theta}\Vert(\widehat{\Gamma}_{EE}-\Gamma_{EE})\Vert_{\max}+\Vert\widehat{K}_E-K_E\Vert_\infty+\lambda_n/\sqrt{m}\right)$. If we can show $\Vert F(\Delta)\Vert_\infty \leq\tilde{r}$ for each $\Vert\Delta\Vert_\infty\leq\tilde{r}$, from Brouwer's fixed point theorem \citep{ortega2000iterative}, it follows that some fixed point satisfies $\Vert\widehat{\Delta}\Vert_\infty\leq\tilde{r}$. For $\Vert\Delta\Vert_\infty\leq\tilde{r}$,

\begin{align}
	\Vert F_{ij}\Vert_2 &\leq \frac{\kappa_\Gamma}{\sqrt{m(d+\left| E\right|)}}\left(\Vert (\widehat{\Gamma}_{EE}-\Gamma_{EE})(\Delta_E+\theta^*_E)\Vert_2 + \Vert\widehat{K}_{E}-K_{E}\Vert_2 + \lambda_n\Vert\widehat{Z}_E\Vert_2\right) \\
	&\leq \kappa_\Gamma \left(\Vert (\widehat{\Gamma}_{ij,E}-\Gamma_{ij,E})(\Delta_E+\theta^*_E)\Vert_\infty+\Vert\widehat{K}_{E}-K_{E}\Vert_\infty + \lambda_n/\sqrt{m}\right).
\end{align}

Now,
\begin{align}
	\Vert(\widehat{\Gamma}_{EE}-\Gamma_{EE})\Delta\Vert_\infty &\leq \left\{\max_{i\in V}\sum_{j\in V, k\leq m} \left|\Delta_{ij}^k\right|\right\}\cdot  \Vert\widehat{\Gamma}_{EE}-\Gamma_{EE}\Vert_{\max} \\
	&\leq ms \Vert\Delta\Vert_\infty \Vert\widehat{\Gamma}_{EE}-\Gamma_{EE}\Vert_{\max},
\end{align}

and similarly,

\begin{align}
	\Vert(\widehat{\Gamma}_{EE}-\Gamma_{EE})\theta^*_E\Vert_\infty &\leq 2\kappa_{1,\theta}\Vert\widehat{\Gamma}_{EE}-\Gamma_{EE}\Vert_\infty.
\end{align}

Thus, if $\Vert\widehat{\Gamma}_{EE}-\Gamma_{EE}\Vert_{\max}\leq\frac{1}{2ms\kappa_\Gamma}$,
\begin{align}
	\Vert F\Vert_\infty \leq \max_{ij}\Vert F_{ij}\Vert_2 &\leq \kappa_\Gamma\left(\Vert\widehat{\Gamma}_{EE}-\Gamma_{EE}\Vert_{\max} (2\kappa_{1,\theta}+ms\tilde{r})+\Vert\widehat{K}-K\Vert_\infty+\lambda_n/\sqrt{m}\right) \\
	&\leq \frac{\tilde{r}}{2}+\frac{\tilde{r}}{2}\leq \tilde{r}.
\end{align}
\end{proof}

\begin{lem}\label{primaldual2}
Suppose that $\sqrt{m}\Vert\widehat{K}-K\Vert_\infty \leq\frac{\tau\lambda_n}{4}$,  $m^{1/2}\Vert\widehat{\Gamma}-\Gamma\Vert_{\max}(\kappa_\theta+s\sqrt{m}\lambda_n)\leq\frac{\tau\lambda_n}{4}$, and $\lambda_n/\sqrt{m}\geq 2\kappa_\Gamma(ms\kappa_\theta\Vert\widehat{\Gamma}_{EE}-\Gamma_{EE}\Vert_{\max}+\Vert\widehat{K}-K\Vert_\infty)$. Then for each $(i,j)\in E^c$,

\begin{align}
	\Vert\widehat{Z}_{ij}\Vert_2<1.
\end{align}
\end{lem}
\begin{proof}

For $(i,j)\in E^c$, the stationary conditions are

\begin{align}
	0&=\widehat{\Gamma}_{E^cE}\widehat{\theta}_E+\widehat{K}_{E^c}+\lambda_n\widehat{Z}_{E^c} \\
	&= \Gamma_{E^cE}(\widehat{\theta}_E-\theta^*_E) +(\widehat{\Gamma}_{E^cE}-\Gamma_{E^cE})\widehat{\theta}_E \\
	&\qquad+\widehat{K}_{E^c}-K_{E^c}+\lambda_n\widehat{Z}_{E^c},\notag
\end{align}

re-arranging and plugging in the stationary conditions for $\widehat{\theta}_E$, we have for $(i,j)\in E^c$,

\begin{align}
	\widehat{Z}_{ij} &= \frac{1}{\lambda_n} \bigg\{-\Gamma_{ij,E}\Gamma_{EE}^{-1}(-(\widehat{\Gamma}_{EE}-\Gamma_{EE})\widehat{\theta}_E-(\widehat{K}_E-K_E)-\lambda_n\widehat{Z}_E) \\
	&\qquad\qquad-(\widehat{\Gamma}_{ij,E}-\Gamma_{ij,E})\widehat{\theta}_E-\widehat{K}_{ij}+K_{ij}\bigg\}.\notag
\end{align}

Applying the $L_2$ norm,

\begin{align}
	\Vert\widehat{Z}_{ij}\Vert_2 &=\frac{1}{\lambda_n}\bigg\{\Vert\Gamma_{ij,E}\Gamma_{EE}^{-1}\Vert_2(\Vert(\widehat{\Gamma}_{EE}-\Gamma_{EE})\widehat{\theta}_E\Vert_2 + \Vert\widehat{K}_E-K_E\Vert_2+\lambda_n\Vert\widehat{Z}_E\Vert_2)\\
	&\qquad+\Vert\widehat{\Gamma}_{ij,E}-\Gamma_{ij,E}\widehat{\theta}_E\Vert_2+\Vert\widehat{K}_{ij}-K_{ij}\Vert_2 \bigg\}\notag\\
	&\leq\frac{1}{\lambda_n}\left\{ \sqrt{m}(2-\tau)(\Vert\widehat{K}-K\Vert_\infty+\Vert(\widehat{\Gamma}-\Gamma)\widehat{\theta}_E)\Vert_\infty\right\}+1-\tau.
\end{align}

Observe that since $\Vert\widehat{\theta}-\theta^*\Vert_{\max}\leq\tilde{r}$,

\begin{align}
	\Vert(\widehat{\Gamma}-\Gamma)\widehat{\theta}\Vert_\infty &
	\leq \Vert\widehat{\Gamma}-\Gamma\Vert_{\max}(\kappa_{1,\theta}+ms\tilde{r}) \\
	&\leq \Vert\widehat{\Gamma}-\Gamma\Vert_{\max}\left(\kappa_{1,\theta}+2s\sqrt{m}\lambda_n\right),
\end{align}

so $\Vert\widehat{Z}_{ij}\Vert_2$ is bounded by
\begin{align}
	\frac{1}{\lambda_n}\left\{(2-\tau)\sqrt{m}\Vert\widehat{K}-K\Vert_\infty+\sqrt{m}(2-\tau)\Vert\widehat{\Gamma}-\Gamma\Vert(\kappa_{1,\theta}+2s\sqrt{m}\lambda_n)\right\}+1-\tau.
\end{align}

if $\sqrt{m}\Vert\widehat{K}-K\Vert_\infty\leq\frac{\tau\lambda_n}{4}$ and $\sqrt{m}\Vert\widehat{\Gamma}-\Gamma\Vert_{\max}(2\kappa_{1,\theta}+2s\sqrt{m}\lambda_n)\leq\frac{\tau\lambda_n}{4}$, this is bounded by
\begin{align}
	(2-\tau)\left(\frac{\tau}{4}+\frac{\tau}{4}\right)+1-\tau\leq 1-\frac{\tau}{2}<1.
\end{align}
\end{proof}

\begin{proof}[Proof of Theorem \ref{modelselectthm}]
Using a concentration bound for $\widehat{K}_{ij}^u-K_{ij}^u$ and applying a union bound, we have that when $t\leq \nu_1$ for some $\nu$, $\Vert\widehat{K}-K\Vert_{\max}>t$ with probability no more than $\exp\{2\log(md)-c_2nt^2\}$ for a constant $c_2$. Similarly,
\begin{align}
	\Vert(\widehat{\Gamma}-\Gamma)\theta^*\Vert_{\max}&\leq 2\kappa_{1,\theta}\Vert\widehat{\Gamma}-\Gamma\Vert_{\max},
\end{align}

and for $t\leq \nu_2$, $\Vert\widehat{\Gamma}-\Gamma\Vert_{\max}>t$ with probability no more than $\exp\{2\log(md)-c_1nt^2\}$. Thus setting $\lambda_n= C\sqrt{\frac{m\kappa_{1,\theta}^2\log(md)}{n}}$ for sufficiently large $C$, and if $\sqrt{\frac{m^2s^2\log md}{n}}=o(1)$, the assumptions of lemma \ref{primaldual2} will be satisfied with probability approaching one. Further, assumption \eqref{pdscm1} is satisfied when $\Vert\widehat{\theta}-\theta^*\Vert_\infty\leq\frac{\rho^*}{2}$. Since $\Vert\widehat{\theta}-\theta^*\Vert_\infty = O(\lambda_n/\sqrt{m})$, we require $\frac{\lambda_n}{\rho^*\sqrt{m}}=o(1)$.

\end{proof}

\begin{lem} \label{bonnetlem}
Let $\phi_k$ be the $k$th orthonormal Legendre polynomial on $[0,1]$. then
\begin{align}
	\left| x(1-x)\frac{\partial \phi_k(x)}{\partial x}\right| &= O(k^{3/2}),\\
	\left| x(1-x)\frac{\partial^2\phi_k(x)}{\partial x^2}\right| &= O(k^{5/2}).
\end{align}

\end{lem}
\begin{proof}
	From Bonnet's recursion formula \citep{abramowitz1964handbook},
	\begin{align}
		x(1-x)\frac{d\phi_k(x)}{dx} &= \frac{k}{2}\left((2x-1)\phi_k(x)-\sqrt{\frac{2k+1}{2k-1}}\phi_{k-1}(x)\right), \label{bonnet}
	\end{align}
	so taking absolute values of each side, and using $\left|\phi_k\right| \leq \sqrt{2k+1}$,
	\begin{align}
		\left| x(1-x)\frac{d\phi_k}{dx}\right| &\leq k\left(\left|\phi_k\right|+\left|\phi_{k-1}\right|\right)\\
		&= O\left(k^{3/2}\right). \label{derivbound}
	\end{align}
	
	Now, using Legendre's differential equation \citep{abramowitz1964handbook},
	\begin{align}
		4x(1-x)\frac{d^2\phi_k}{dx^2}+2(2x-1)\frac{d\phi_k}{dx}-k(k+1)\phi_k &=0,
	\end{align}
	and using the fact that $\frac{d\phi_{k}}{dx}\leq \frac{k(k+1)\sqrt{2k+1}}{2}$, we find that
	\begin{align}
		\left|x(1-x)\frac{d^2\phi_k}{dx^2}\right| &\leq \frac{1}{2}\left|\frac{d\phi_k}{dx}\right|+\frac{1}{4}k(k+1)\left|\phi_k\right|\\
		&= O\left(k^{5/2}\right).
	\end{align}
	
\end{proof}

\begin{proof}[Proof of Theorem \ref{msnonp}]
The proof technique is essentially the same as Theorem \ref{modelselectthm} so we omit some details. The main difference is that here $K^*= - \Gamma^*\theta^*-\Gamma_{T}\theta_T$, so we must deal with one additional term in the analysis, the bias from truncation. Suppose $m_1=m_2$. We choose $\lambda_n$ so that with high probability,
\begin{align}
	\lambda_n &=\Omega\left( m_2\kappa_{1,\theta}\Vert\widehat{\Gamma}_{EE}-\Gamma_{EE}\Vert_{\max}+m_2\Vert\widehat{K}-K\Vert_{\max}+\kappa_Tm_2\Vert\theta_T\Vert_{\max}\right),\\
	\lambda_n &\rightarrow 0,
\end{align}

as well as requiring $n=\Omega(m_2^4 s^2 \log md)$.

Now, applying Lemma \ref{bonnetlem},  $\Vert A(x)\Vert_{\max} = O\left(m_2^4\right)$, so applying Hoeffding's inequality and a union bound as well as the boundedness assumption \ref{assump4}, $\Vert\widehat{\Gamma}_{EE}-\Gamma_{EE}\Vert_{\max}\leq C\sqrt{\frac{m_2^8 \log(m_2 d)}{n}}$ with probability approaching one, for sufficiently large constant $C$. Similarly, $\Vert\widehat{K}-K\Vert_{\max}\leq C'\sqrt{\frac{m_2^6\log(m_2d)}{n}}$ with probability approaching one. Furthermore, $\Vert\theta_T\Vert_{\max}=O(m_2^{-r-1/2})$. Thus, supposing $\kappa_{1,\theta}=O(m_2^2)$, we need
\begin{align}
	\lambda_n \asymp O\left(\sqrt{\frac{m_2^{12}\log(m_2d)}{n}} + m_2^{-r+1/2}\right).
\end{align}

Balancing the two terms, we choose $m_2\asymp n^{\frac{1}{2r+13}}$, so $\lambda_n\asymp \sqrt{\frac{\log(nd)}{n^{\frac{2r-1}{2r+13}}}}$. The stated sample complexity ensures that $\lambda_n\rightarrow 0$. Furthermore, $\Vert\widehat{\theta}_E-\theta^*_E\Vert_{\max}=O\left(\lambda_n/m_2\right)$, so we require $\frac{\lambda_n}{m_2\rho^*}\rightarrow 0$.

\end{proof}

\bibliographystyle{apalike}

\bibliography{quasr.bib}

\end{document}